\documentclass[11pt]{article}

\usepackage{fancyhdr}
\usepackage{amsfonts}
\usepackage{amsmath}
\usepackage{amssymb}
\usepackage{amsthm}
\usepackage{tikz}
\usepackage{accents}

\usepackage{mathrsfs}
\usepackage{amscd}
\usepackage{amstext}
\usepackage{authblk}

\usepackage{graphics}
\usepackage{graphicx}

\newlength{\fixboxwidth}
\newcommand{\re}{\mathbb{R}}\newcommand{\N}{\mathbb{N}}
\newcommand{\zz}{\mathbb{Z}}\newcommand{\C}{\mathbb{C}}

\newcommand{\Z}{{\zz}^d}

\newcommand{\R}{{\re}^d}

\newcommand{\cs}{{\mathcal S}}

\newcommand{\cf}{{\mathcal F}}
\newcommand{\cfi}{{\cf}^{-1}}

\newcommand{\trace}{{\rm tr \, }}

\newcommand{\supp}{{\rm supp \, }}

\newcommand{\mix}{{\rm mix}}
\newcommand{\gf}{\mathcal{F}}
\newcommand{\unif}{{\rm unif}}

\newcommand{\bproof}{\begin{proof}}
\newcommand{\eproof}{\end{proof}}

\newcommand{\be}{\begin{equation}}
\newcommand{\ee}{\end{equation}}
\newcommand{\beq}{\begin{eqnarray}}
\newcommand{\beqq}{\begin{eqnarray*}}
\newcommand{\eeq}{\end{eqnarray}}
\newcommand{\eeqq}{\end{eqnarray*}}

\numberwithin{equation}{section}

\newtheorem{theorem}{Theorem}[section]
\newtheorem{definition}[theorem]{Definition}

\newtheorem{lemma}[theorem]{Lemma}
\newtheorem{proposition}[theorem]{Proposition}
\newtheorem{remark}[theorem]{Remark}

\begin{document}

\title{Pointwise Multipliers for Sobolev and Besov Spaces of Dominating Mixed Smoothness}

\author[a,b]{Van Kien Nguyen\thanks{E-mail: kien.nguyen@uni-jena.de,\ kiennv@utc.edu.vn}}
\author[a]{Winfried Sickel\thanks{E-mail: winfried.sickel@uni-jena.de}}
\affil[a]{Friedrich-Schiller-University Jena, Ernst-Abbe-Platz 2, 07737 Jena, Germany}
\affil[b]{University of Transport and Communications, Dong Da, Hanoi, Vietnam}


\date{\today}

\maketitle
\begin{abstract}
Under certain restrictions we describe the set of all pointwise multipliers in case of Sobolev  and Besov spaces of 
dominating mixed smoothness.
In addition we shall give necessary and sufficient conditions for the case that these spaces form algebras with 
respect to pointwise multiplication.
\end{abstract}


\section{Introduction}


The regularity concepts related to  Sobolev and Besov spaces of dominating mixed smoothness 
are standard in {\em Approximation Theory} \cite{T93b}, {\em Numerical Analysis} \cite{BG}, \cite{SST08}  and 
{\em Information-Based Complexity} 
\cite{NoWo08}, \cite{NoWo10}, \cite{NoWo12}.
However, there is also some interest in {\em Learning Theory} in those classes, at least in $S^m_2 W(\R)$, $m \in \N$, and 
$S^r_{2,2}B(\R)$, $r>0$, see \cite{sc}, \cite{lki}. Recently we have been asked by Lev Markhasin and Ingo Steinwarth 
about pointwise multipliers for those classes. By dealing with this problem it turned out that these problems become more difficult  
compared to the isotropic situation.
It will be the aim of this paper to describe this in detail.
\\

As it is well-known, Sobolev spaces $W^m_{p}(\R)$ form an algebra with respect to pointwise multiplication if $m>d/p$.
This means there exists a constant $c_1$ such that
\be\label{ws-00}
\| \, f\cdot g\, |W^m_{p}(\R)\| \le c_1 \,  \| \, f\, |W^m_{p}(\R)\|\,  \| \,  g\, |W^m_{p}(\R)\|
\ee
holds for all $f,g \in W^m_p (\R)$.
In this paper we shall deal with a generalization of this fact to tensor product Sobolev spaces $S^m_p W(\R)$ where
\beqq
S^m_p W (\R)  & := &   \underbrace{W^m_{p}(\re) \otimes_{\alpha_p} W^m_{p}(\re) \otimes_{\alpha_p}   
\ldots   \otimes_{\alpha_p}\,  W^m_{p}(\re)} \, .
\\
& & \qquad \qquad\qquad \quad\mbox{d times}
\eeqq
Here $1 < p< \infty$ and $\alpha_p $ denotes the $p$-nuclear norm. 
For a moment we concentrate on the two-dimensional case.
Using the cross-norm property of $S^m_p W(\re^2) $ and \eqref{ws-00} we conclude for tensor products
$f= f_1 \otimes f_2 $ and $g= g_1 \otimes g_2$ with $f_1,f_2, g_1, g_2  \in W^m_p (\re)$ that  
\beq\label{ws-01}
\| \, f\cdot g\, | S^m_{p}W(\re^2)\| & = &  \| \, (f_1 \cdot g_1) \otimes (f_2 \cdot g_2)\,  |S^m_p W (\re^2) \|
= \| \, f_1 \cdot g_1\, | W^m_p (\re) \|\cdot \| \,  f_2 \, \cdot \,  g_2\,  |W^m_p (\re) \|
\nonumber
\\
&\le & c^2_1 \,  \| \, f_1\, |W^m_{p}(\re)\|\,  \| \,  g_1\, |W^m_{p}(\re)\|\, 
\, \| \, f_2\, |W^m_{p}(\re)\|\,  \| \,  g_2\, |W^m_{p}(\re)\|
\nonumber
\\
& =  & c^2_1 \,  \| \, f_1 \otimes f_2 \, |S^m_{p}W(\re^2)\|\,  \| \,  g_1 \otimes g_2\, |S^m_{p}W(\re^2)\|\, 
\nonumber
\\
& =  & c^2_1 \,  \| \, f \, |S^m_{p}W(\re^2)\|\,  \| \,  g\, |S^m_{p}W(\re^2)\|\, . 
\eeq
Here $c_1$ is the same constant as in \eqref{ws-00} for $d=1$.
Since $\|\, \cdot \, |S^m_p W (\re^2)\|$ is an uniform cross-norm it follows that in this particular situation where $f$ is given by a tensor product
the linear operator $T_f: ~ g \mapsto f \, \cdot \, g$ maps $S^m_p W(\re^2)$ into itself, see, e.g.,  \cite[Lemma~1.30]{LiCh}.
Hence, any operator $T_f$, where  
\be\label{ws-02q}
f= \sum_{k=1}^N f_{k,1} \otimes f_{k,2}\, , \qquad 
f_{k,1},\, f_{k,2} \in W^m_p (\re) \, , 
\ee
maps  $S^m_p W(\re^2)$ into itself.
The set of functions defined in \eqref{ws-02q} forms a dense set in $S^m_p W(\re^2)$. However, the present situation does not allow 
to conclude that all functions $f\in S^m_p W (\re^2)$ 
generate an operator $T_f$ which maps $S^m_p W(\re^2)$ into itself. As a consequence of 
\eqref{ws-01} we only get 
\[
 \| \, f\cdot g\, | S^m_{p}W(\re^2)\| \le  c^2_1 \, \Big(\sum_{k=1}^N  \,  \| \, f_{k,1} \otimes f_{k,2} \, |S^m_{p}W(\re^2)\|\Big)\,  
\| \,  g\, |S^m_{p}W(\re^2)\|\, 
\]
In what follows we will improve this estimate to 
\[
 \| \, f\cdot g\, | S^m_{p}W(\R)\| \le  c_2 \,  \| \, f \, |S^m_{p}W(\R)\|\,  
\| \,  g\, |S^m_{p}W(\R)\|\, \, , \qquad f,g \in S^m_p W (\R)\, , 
\]
mainly based on improved trace inequalities.
It is well-known that the mapping $\trace : ~f(x_1, x_2)\mapsto f(x_1,0) $ maps $S^m_p W (\re^2)$ continuously onto 
$W^m_p (\re)$. Let $d= d_1 + d_2$, $d_1,d_2 \in \N$,  and $\alpha = (\alpha_1, \ldots \alpha_{d_1}, 0, \ldots \, , 0) \in \N_0^d$, 
$\max_\ell |\alpha_\ell| \le m$.
We shall show below that in the general $d$-dimensional context we always have an inequality of the form
\[
\| \, \sup_{y \in \re^{d_2}} |D^\alpha f(x,y)|\, |L_p (\re^{d_1})\| \le c_3\, 
\|\, f \, |S^m_p W (\R)\|
\]
with $c_3$ independent of $f$.\\

In a similar way we shall proceed in case of  Besov spaces. Let $r$ be a positive real number  and $1\le p\le \infty$.
We define
\beqq
S^r_{p,p} B (\R)  & := &   \underbrace{B^r_{p,p}(\re) \otimes_{\alpha_p} B^r_{p,p}(\re) \otimes_{\alpha_p}   
\ldots   \otimes_{\alpha_p}\,  B^r_{p,p}(\re)} \, .
\\
& & \qquad \qquad\qquad \quad\mbox{d times}
\eeqq
For those tensor product Besov spaces, in case $r>1/p$, we shall show the inequality  
\be\label{ws-013}
 \| \, f\cdot g\, | S^r_{p,p} B(\R)\| \le  c_4 \,  \| \, f \, |S^r_{p,p}B(\R)\|\,  
\| \,  g\, |S^r_{p,p}B(\R)\|\, \, , \qquad f,g \in S^r_{p,p} B (\R)\, , 
\ee 
In both cases, the Sobolev spaces of dominating mixed smoothness $S_p^m W(\R)$ and 
the Besov spaces of dominating mixed smoothness $S^r_{p,p} B(\R)$, we also able to describe the set of all pointwise multipliers
$M(S_p^m W(\R))$ and $M(S_{p,p}^r B(\R))$, respectively. Our  proof of \eqref{ws-013} relies on the characterization of $S^r_{p,p} B(\R)$ by differences as in the 
classical paper \cite{Str-67} of Strichartz or in the monographs \cite{MS1}, \cite{MS2} by Maz'ya and Shaposnikova.
It seems that the method of using paraproducts, already applied in Peetre \cite{Pe}, Triebel \cite{Tr78}, \cite{Tr83} or Runst, S. \cite{RS}, is less convenient in 
the context of dominating mixed smoothness. 
\\

Probably less well-known is the fact that the intersections $W^m_{p}(\R)\cap L_\infty (\R)$ and
$B^r_{p,p}(\R)\cap L_\infty (\R)$
form algebras with respect to pointwise multiplication. Here   $m \in \N_0$ and $r>0$.
We refer to Moser \cite{Mo}, Zeidler \cite[Prop.~21.77]{Ze} for the Sobolev case with $p=2$,  Peetre \cite[Thm.~11, p.~147]{Pe}
for the Besov case and to \cite[Thm.~4.6.4/2]{RS} for the general situation.
In addition the following Moser-type inequalities hold
\[
\| \, f\cdot g\, |W^m_{p}(\R)\| \le c_5 \, \big( \| \, f\, |W^m_{p}(\R)\|\, \| \,  g\, |L_{\infty}(\R)\| + \| \,  g\, |W^m_{p}(\R)\|\, \| \, f\, |L_{\infty}(\R)\|\big)
\]
for all $f,g \in W^m_p (\R) \cap L_\infty (\R)$ and
\[
\| \, f\cdot g\, |B^r_{p,p}(\R)\| \le c_6 \, \big( \| \, f\, |B^r_{p,p}(\R)\|\, \| \,  g\, |L_{\infty}(\R)\| + \| \,  g\, |B^r_{p,p}(\R)\|\, \| \, f\, |L_{\infty}(\R)\|\big)
\]
for all $f,g \in B^r_{p,p} (\R) \cap L_\infty (\R)$, respectively.

To our own surprise these inequalities do not have a counterpart in the tensor product situation.
\\

The paper is organized as follows.
In the next Section \ref{def} we shall collect all what is needed
about these tensor product  function spaces. Mainly we shall work with Fourier analytic descriptions and characterizations by differences  of these classes.
In Section \ref{main} we shall state and comment on our main results.
All proofs are concentrated in Section \ref{proofs}.


\subsection*{Notation}


As usual $\N$ denotes the natural numbers, $\N_0=\N\cup\{0\}$, 
$\zz$ denotes the integers, 
$\re$ the real numbers, 
and $\C$ the complex numbers. The letter $d\in \N, \ d>1,$ is always reserved for the underlying dimension in $\R, \Z$ etc and by $[d]$ we mean
$[d]=\{1,...,d\}$. If $k=(k_1, ... , k_d)\in \N_0^d$, then we put
\beqq
|k|_1 := k_1 + \ldots \, + k_d\, \qquad\text{and}\qquad |k|_\infty:= \max_{j=1, \ldots ,\, d} \, k_j \,.
\eeqq
We denote
with $\langle x,y\rangle$ or $x\cdot y$ the usual Euclidean inner product in $\R$. By $x\diamond y$ we mean
\beqq
x\diamond y =(x_1y_1,...\,, x_dy_d)\in \R\,.
\eeqq
For a subset $e$ of $\{1,2,\ldots \, , d\}$ we put
\[
\N_0^d(e) := \big\{k\in \N_0^d: \ k_i=0\ \text{if}\ i\not \in e\big\}.
\]
If $X$ and $Y$ are two (quasi-)normed spaces, the (quasi-)norm
of an element $x$ in $X$ will be denoted by $\|x\,|\,X\|$. 
The symbol $X \hookrightarrow Y$ indicates that the
identity operator is continuous. For two sequences $a_n$ and $b_n$ we will write $a_n \lesssim b_n$ if there exists a
constant $c>0$ such that $a_n \leq c\,b_n$ for all $n$. We will write $a_n \asymp b_n$ if $a_n \lesssim b_n$ and $b_n
\lesssim a_n$. 

Let $\cs(\R)$ be the Schwartz space of all complex-valued rapidly decreasing infinitely differentiable  functions on $\R$. 
The topological dual, the class of tempered distributions, is denoted by $\cs'(\R)$ (equipped with the weak topology).
The Fourier transform on $\cs(\R)$ is given by 
\[
\cf \varphi (\xi) = (2\pi)^{-d/2} \int_{\R} \, e^{-ix \xi}\, \varphi (x)\, dx \, , \qquad \xi \in \R\, .
\]
The inverse transformation is denoted by $\cfi $.
We use both notations also for the transformations defined on $\cs'(\R)$\,.
\\


\section{Sobolev and Besov spaces of dominating mixed smoothness}\label{def}


For our methods the tensor product approach to these function spaces is not appropriate.
We shall introduce them by derivatives and differences.


\subsection{Sobolev spaces of dominating mixed smoothness}


The interpretation of  Sobolev spaces of dominating mixed smoothness as tensor product spaces is taken from 
\cite{SU09,SU10}.
We refer also to these papers for a definition of $X \otimes_{\alpha_p} Y$. However, here we shall work with the following.

\begin{definition}\label{def-so} Let $1< p< \infty$ and $m\in \N$. Then the  Sobolev space of dominating mixed 
smoothness $S^m_p W(\R)$ is the collection of all functions 
$f \in L_p (\R)$ such that all distributional derivatives $D^\alpha f$ with $\max_{j=1, \ldots \, d} \, \alpha_j\le m$
belong to $L_p (\R)$. We put
\beqq
\|\, f\, |S^m_pW(\R)\|^{(d)} := \sum_{|\alpha|_{\infty} \leq m} \, \|\, D^{\alpha}f\, |L_p(\R)\|\,.
\eeqq 
\end{definition}

\begin{remark} \rm
  Let $1<p<\infty$ and $m\in \N_0$. By $W^m_p(\R)$ we denote the isotropic classical Sobolev spaces equipped with the norm
\beqq
\|f|W^m_p(\R)\|:= \sum_{|\alpha|_{1} \leq  m}\|D^{\alpha}f|L_p(\R)\|\, .
\eeqq 
$\|\, \cdot\, |S^m_pW(\R)\|$ is a cross-norm,  i.e., 
if $f_i \in W^m_p(\re),\ i=1, \ldots, d,$ then 
\[
 f(x) = \prod_{i=1}^d f_i (x_i)  \in S^m_pW(\R) \qquad \text{and} \qquad
 \| \, f \, | S^m_pW(\R)\| = \prod_{i=1}^d \|\, f_i \, |W^m_p(\re)\| \, .
\]
In case $d=1$ we have $S^m_p W (\re)= W^m_p (\re)$.
\end{remark}

As in case of isotropic Sobolev spaces it will be enough to concentrate on the $L_p$-norms 
of the function itself and of those derivatives with the highest order, i.e., those derivatives  $D^\alpha$, where 
$ \alpha\in \{0,m\}^d$, see \cite{ST} (combine Def. 2.3.1 and Theorem 2.3.1) or \cite[Cor.~2.1.1]{Hansen}.

\begin{lemma}\label{equi} Let $1< p< \infty$ and $m\in \N_0$. Then $S^m_p W(\R)$ is the collection of $f\in L_p(\R)$ such that
\beqq
\|f|S^m_pW(\R)\|^*:= \sum_{ \alpha\in \{0,m\}^d}\|D^{\alpha}f|L_p(\R)\|<\infty \, .
\eeqq
$\|\, \cdot \, |S^m_pW(\R)\|^*$  and $\|\, \cdot \, |S^m_pW(\R)\|^{(d)}$ are equivalent.
\end{lemma}

Of some importance will be embeddings into the class $C(\R)$ of all bounded and continuous functions 
equipped with the supremum norm.

\begin{lemma}\label{emb}
Let $m \in \N$ and $1 < p < \infty$. Then
the space $S^m_p W(\R)$ is continuously embedded into $C (\R)$.
\end{lemma}

For a proof we refer to \cite[Remark~2.4.1/2]{ST}.


\subsection{Besov spaces of dominating mixed smoothness}


Next we shall give the definition of Besov spaces of dominating mixed smoothness. Therefore we use differences. 
But before doing that we recall the definition of (isotropic) Besov spaces. \\
 
For a multivariate function $f:\R\to \C$,  $m  \in \N$, $h  \in \R$ and $x \in \R$ we put
 \[
 \Delta_{h}^{m} f(x):= \sum_{\ell =0}^{m} (-1)^{m -\ell} \, \binom{m}{\ell} \, f(x + \ell h )
 \]
and
\[
 \omega_m (f,t)_p := \sup_{|h|<t} \|\, \Delta_{h}^{m} f\, |L_p (\R)\|\, , \qquad t>0\, .
\]
Let $1\leq p\leq \infty$  and $r>0$, $m-1 \le r <m$. Then the (isotropic) Besov space 
$B^r_{p,p}(\R)$ is a collection of all $f\in L_p(\R)$ such that 
\be\label{besov}
\|f|B^r_{p,p}(\R)\|:= \|\, f\, |L_p(\R)\| + \Big( \sum_{j=0}^{\infty} \, (2^{jr}\,  \omega_m (f, 2^{-j}))^p\Big)^{1/p} < \infty. 
\ee
Clearly, in a similar way one could  define the more general spaces $B^r_{p,q}(\R)$, $1\le q \le \infty$, however,
we will not need them here.
\\

\noindent
Now we turn to Besov spaces of dominating mixed smoothness. 
Let $i \in [d]$, $m \in \N$, $h \in \re$ and $x \in \R$ we put
\beqq
 \Delta_{h,j}^{m} f(x):= \sum_{\ell =0}^{m} (-1)^{m-\ell} \, \binom{m}{\ell} \, 
f(x_1, \ldots  , x_{j-1}, x_j + \ell h, x_{j+1}, \ldots  , x_d)\, .
\eeqq
This is the $m$-th order difference of $f$ in direction $j$.  
For $e\subset [d]$, $ h  \in \R$ and $m \in \N_0^d$ the mixed $(m,e)$-th difference operator $\Delta_h^{m,e}$ is defined to be 
 \[
 \Delta_{ h}^{m,e} := \prod_{i \in e} \Delta_{h_i,i}^{m_i} \quad\mbox{and}\quad \Delta_h^{m,\emptyset} :=  \operatorname{Id} \,, 
 \]
where $\operatorname{Id}f = f$. An associated modulus of smoothness is given by
\beqq
 \omega_{m}^e(f,t)_p:= \sup_{|h_i| < t_i, i \in  e}\|\, \Delta_h^{m,e}f \, |L_p(\R)\| \quad,\quad t \in [0,1]^d\,
\eeqq
 for $f \in L_p(\R)$ (in particular, $\omega_{m}^{\emptyset}(f,t)_p = \|f|L_p(\R)\| $). 
Many times, e.g., in the definition below we do not need to choose  $m$ as a vector.
For this reason, if $m= (n,\ldots\, , n)$ we put 
\[
\overline{n}:= (n,\ldots\, , n)\, , \qquad n \in \N\, .
\]
For a set $e\subset [d]$ we denote $$\N_0^d(e)=\big\{k\in \N_0^d: \ k_i=0\ \text{if}\ i\not \in e\big\}.$$
Let $k \in \N_0^d$.
For brevity we write $2^{-k} $ instead of the vector $(2^{-k_1}, 2^{-k_2}, \ldots \, , 2^{-k_d})$

\begin{definition}\label{diff} 
Let $1\leq p \leq \infty$, $r>0$ and $m-1 \le r < m$ for some  $m \in \N$. 
Then the Besov space of dominating mixed smoothness $S^r_{p,p}B(\R)$
is the collection of all $f\in L_p(\R)$ such that
$$
    \|\, f \, | S^r_{p,p}B(\R)\|^{(m)} := 
\sum_{e\subset [d]}\Big(\sum\limits_{k\in \N_0^d(e)} 2^{r|k|_1 p}\omega_{\overline{m}}^{e }(f,2^{-k})_p^{p}\Big)^{1/p}\, 
$$
is finite with the usual modification if $p=\infty$. 
\end{definition}

\begin{remark}\label{blabla} \rm 
(i) If $d=1$ we get $ S^{r}_{p,p} B(\re) = B^r_{p,p}(\re)$.  
\\
(ii)
Besov spaces of dominating mixed smoothness also have a cross-norm. If $f_i \in B^r_{p,p}(\re)$, $ i=1, ... , d$,
 then its tensor product 
\[
 f(x) :=  (f_1 \otimes f_2 \otimes \, \ldots \, \otimes f_d ) (x) = \prod_{i=1}^d f_i (x_i)\, , \qquad x = (x_1, \, \ldots \, , x_d) \in \R\, ,
 \]
belongs to $S^r_{p,p}B(\R)$ and  
\[
 \| \, f \, | S^r_{p,p}B(\R)\| = \prod_{i=1}^d \|\, f_i \, |B^r_{p,p} (\re)\| \, .
\]
(iii) For the interpretation of $ S^{r}_{p,p} B(\R)$ as tensor products of $B^r_{p,p}(\re)$ we refer to 
\cite{SU09,SU10}.
\end{remark}

Next we recall two properties which will be of certain use later on.

\begin{lemma}\label{diff1}
Let $r >0$ and $1 \le  p \le \infty$. 
Let $m \in \N_0^d$ such that $r < m_i$ for all $i\in [d]$.
Then
$$
   \|\, f \, | S^r_{p,p}B(\R)\|^{(m)} := 
\sum_{e\subset [d]}\Big(\sum\limits_{k\in \N_0^d(e)} 2^{r|k|_1 p}\omega_{{m}}^{e }(f,2^{-k})_p^{p}\Big)^{1/p}\, 
$$
is an equivalent norm on the space $S^r_{p,p} B(\R)$. 
\end{lemma}

For a proof we refer to \cite[2.3.4]{ST} ($d=2$) and \cite{U1}.

\begin{lemma}\label{emb1}
Let $r >0$ and $1 \le  p \le \infty$. Then
the space $S^r_{p,p} B(\R)$ is continuously embedded into $C (\R)$ if and only if either $r>1/p$ or $r=1=p$.
\end{lemma}

For a proof we refer to \cite[2.3.4]{ST} ($d=2$) and \cite{Vybiral}.


\subsection{Tools from Fourier analysis}\label{equi-def}


Littlewood-Paley characterizations will play an important role in our investigations.
\\
Let $\varphi_0 \in C_0^{\infty}({\re})$ be a non-negative function such that 
 $\varphi_0(\xi) = 1$ on $[-1,1]$ and $\supp\varphi_0 \subset [-\frac{3}{2},\frac{3}{2}]$. 
For $j\in \N$ we define
\beqq
         \varphi_j(\xi) = \varphi_0(2^{-j}\xi)-\varphi_0(2^{-j+1}\xi) ,\qquad\ \xi \in \re\, , 
\eeqq
and  
\beqq
\varphi_{k}(x) := \varphi_{k_1}(x_1)\cdot...\cdot
         \varphi_{k_d}(x_d)\, , \qquad  x \in \R, \quad k\in \N_0^d\,. 
\eeqq        
This implies 
\beqq
\sum_{k\in \N_0^d} \varphi_k(x)& = & 1 \qquad  \text{for all}\ x\in \R\, , 
\\
\supp \varphi_k & \subset & \Big\{ x \in \R: 2^{k_\ell-1}\le |x_\ell| \le 3 \, 2^{k_\ell-1}\, , \quad \ell =1, \ldots \, , d\Big\}\,.
\eeqq
We shall call the system $\{\varphi_k\}_{k\in \N_0^d}$ a smooth  dyadic decomposition of unity on $\R$.
Let $\chi_0$ be the characteristic function of $[-1,1]$. Let further $\chi_j$, $j\in \N$, be the 
characteristic function of $[-2^j,-2^{j-1})\cup(2^{j-1}, 2^{j}]$. For $k\in \N_0^d$ we define $\chi_k(x)$, $x\in \R$,
as a  tensor product, i.e.,
\be \label{chi}
\chi_k(x) := \chi_{k_1}(x_1)\cdot\, \ldots \, \cdot \chi_{k_d}(x_d)\,,  \qquad  x \in \R, \quad k\in \N_0^d\,. 
\ee
The system $\{\chi_k\}_{k\in \N_0^d}$ represents a nonsmooth  dyadic decomposition of unity on $\R$.

\begin{proposition}\label{def-do}
Let $\{\varphi_k\}_{k\in \N_0^d}$ be the above system.\\
{(i)} Let $1<p<\infty$ and $m\in \N_0$. Then $S^m_pW(\R)$ is the collection of all tempered distributions $f\in  \mathcal{S}'(\R)$
such that
\[
 \|\, f \, |S^m_pW(\R)\|^{\varphi} :=
\Big\| \Big(\sum\limits_{k\in \N_0^d} 2^{2|k|_1 m }\, \big|\, \cfi[\varphi_{k}\, \cf f](\, \cdot \, )\big|^2 \Big)^{1/2} \Big|L_p(\re^d)\Big\|<\infty.
\]
(ii) Let $1\leq p \le \infty$ and  $r>0$. Then  $ S^{r}_{p,p}B(\re^d)$ is the
         collection of all tempered distributions $f \in \mathcal{S}'(\R)$
         such that
\beqq
          \|\, f \, |S^r_{p,p}B(\R)\|^{\varphi} :=
         \Big(\sum\limits_{k\in \N_0^d} 2^{r|k|_1  p}\, \|\, \cfi[\varphi_{k}\, \cf f]
         |L_p(\re^d)\|^p\Big)^{1/p} <\infty.
\eeqq
(iii) If we replace the smooth system $\{\varphi_k \}_k$   by the nonsmooth $\{\chi_k \}_k$ in (i) and (ii) then  we  obtain  equivalent norms in 
case $1 <p<\infty$ in the corresponding spaces.      
\end{proposition}

\begin{remark}\label{re-chi}\rm 
Concerning a proof of part (i) we refer to \cite[Theorem~2.3.1]{ST}.
For $m=0$ part (i) is just a variant of the  Littlewood-Paley assertions, 
see, e.g.,  Lizorkin \cite{Liz,Li} or Nikol'skij \cite[1.5.6]{Ni}. 
The proof of Proposition \ref{def-do}(ii) can be found in \cite[2.3.3,~2.3.4]{ST} and \cite{U1}, 
see also \cite{NUU}. 
The proof of (iii) is a straightforward modification of a similar assertion in the isotropic case, called Lizorkin representations.
We refer to Lizorkin \cite{Li1} and \cite[2.5.4]{Tr83}.
\end{remark}

Next we will collect some required tools from Fourier analysis. We recall an adapted version of the famous Nikol’skij
inequality, see Uninskij \cite{Un1,Un2},  St\"ockert \cite{St} or \cite[Theorem 1.6.2]{ST}.

\begin{proposition}\label{Nikolski} 
Let $0<p_0\leq p\leq \infty$ and ${\alpha}=(\alpha_1,...,\alpha_d)\in \N_0^d$. 
Let $\Omega=[-b_1,b_1]\times \cdots \times [-b_d,b_d]$, $b_i>0$, $i=1,...,d$. 
Then there exists a positive constant $C$, independent of $(b_1, \ldots , b_d)$, such that 
\beqq
\| D^{\alpha}f|L_{p}(\R)\| \leq C\Big(\prod_{i=1}^d b_i^{\alpha_i+\frac{1}{p_0}-\frac{1}{p}}\Big) \|f|L_{p_0}(\R)\|
\eeqq
holds for all $f\in L_{p_0}(\R) \cap \cs'(\R)$ with $\supp \gf f \subset \Omega$\,.
\end{proposition}

The following construction of a maximal function is essentially  due to Peetre, but based on earlier work of Fefferman and Stein.
Let $a>0$ and  $b=(b_1,...,b_d)$, $b_i>0$, $i=1,...,d$ be fixed. Let $f$ be a regular distribution  such that  $\gf f$ is compactly
supported. We define the Peetre maximal function $P_{b,a}f$ by
\beqq
  P_{b,a}f(x) = \sup\limits_{z\in \R} \frac{|f(x-z)|}{\prod_{i=1}^d(1+|b_iz_i|)^a}\, , \qquad x \in \R\, .
\eeqq

\begin{proposition}\label{peetremax}
Let $1 \le p \leq\infty$ and $\Omega=[-b_1,b_1]\times \cdots \times [-b_d,b_d]$, $b_i>0$, $i=1,...,d$. Let further $a>1/p$. 
Then there exists a positive constant $C$,  independent of $(b_1, \ldots , b_d)$, such that
\beqq
\big\| P_{b,a}f \big|L_p(\R)\big\|\leq C\, \|f |L_p(\R)\|
\eeqq
holds for all $f \in L_p(\R)$ with $\supp (\gf f)\subset \Omega$.
\end{proposition}

For a proof we refer to \cite[Thm.~1.6.2]{Tr83}. 
A very useful  relation between Peetre maximal function and differences is given by the following 
lemma, see \cite{U1} and  \cite[2.3.3]{ST} (two-dimensional case).

 \begin{lemma}\label{1dim-1}
 Let $a>0$ and $m \in \N$.  
 Then there exists a constant $C$
 such that 
\beqq
     |\Delta^m_hf(t)| \leq  C\, \max\{1,|bh|^a\}\, \min\{1,|bh|^m\}\, P_{b,a}f(t)\,.
\eeqq
 holds for all $b \ge 1$, all $h\neq 0$, all $t\in \re$ and all $f\in \cs'(\re)$ satisfying $\supp(\gf f) \subset
 [-b,b]$.  
 \end{lemma}

 Applying the above result iteratively  with respect to components in $e\subset [d]$ we get the following modified version in multivariate situation.
 
 \begin{lemma}\label{ddim-1}
 Let  $a>0$, $e\subset [d]$, $m \in \N_0^d$  and $h =
 (h_1,...,h_d) \in \R$. Let further $f\in
 \mathcal{S}'(\R)$ with $\supp(\mathcal{F} f) \subset Q_{b}$, where
 $$
   Q_{b}:=[-b_1,b_1]\times...\times [-b_d,b_d]\,,\ \ b_i>0,\ \ i=1,...,d.
 $$
 Then there exists a constant $C>0$ (independent of $f$, $b$, $x$
 and $h$) such that
 \begin{equation*}
 |\Delta^{m,e}_h  f(x)|
 \leq C\, \bigg(\prod\limits_{i\in e}\, \max\{1,|b_ih_i|^a\}\, \min\{1,|b_ih_i|^{m_i}\}
 \bigg)\,  P_{b,a} f(x)
 \end{equation*}
 holds for all $x\in \R$. 
 \end{lemma}

Finally, we recall a Fourier multiplier assertion for vector-valued $L_p-$spaces of entire analytic function, see \cite[Proposition 2.3.5]{Hansen} 
or \cite[Theorem 1.10.3]{ST} (two-dimensional case).

 \begin{lemma}\label{mul1}
 Let $0<p< \infty$ and $\Omega=\{\Omega_{\ell}\}_{\ell\in \N_0^d} $ be a sequence of compact subsets of $\R$ given by
 \beqq
 \Omega_{\ell} = [-b_1^{\ell},b_1^{\ell}]\times\cdots \times [-b_d^{\ell},b_d^{\ell}] \,.  
 \eeqq  
 Let $r>\frac{1}{\min(p,2)}+\frac{1}{2}$, $r\in \N$. Then there exists a constant $C>0$ such that 
 \beqq
 \bigg\|\Big(\sum_{\ell\in \N_0^d}  \big|\gf^{-1}  M_{\ell}\gf f_{\ell} \big|^2\Big)^{1/2}\,\bigg|\,L_p(\R)\bigg\| 
\leq C\, \sup_{\ell\in \N_0^d} 
 \, \| M_{\ell}(b^{\ell}\, \cdot)|S^r_2W(\R)\| \,  \Big\|\Big(\sum_{\ell\in \N_0^d}  \big|  f_{\ell} \big|^2\Big)^{1/2}\,\Big|\,L_p(\R)\Big\|
 \eeqq
 holds for all systems $ \{f_{\ell}\}_{\ell}$, satisfying $ \Big(\sum_{\ell\in \N_0^d}  \big|  f_{\ell} \big|^2\Big)^{1/2}\in L_p(\R)$ and  
 $\supp(\gf f_{\ell}) \subset \Omega_{\ell}$, $\ell\in \N_0^d$, and all systems $\{M_{\ell}\}_{\ell} \in S^r_2W(\R)$.
 \end{lemma}


\section{Main results}\label{main}


For a Banach space $X$ of functions  we shall call a function $f$ a pointwise multiplier
if $f \, \cdot \, g \in X$ for all $g \in X$
(this is includes, of course, that the operation $g \mapsto f \, \cdot \, g$ must be well defined for all $g\in X$).
If $X \hookrightarrow L_p (\Omega)$ for some $p$ (here $\Omega$ is a domain in $\R$),  
as a consequence of the Closed Graph Theorem, we obtain that the liner operator 
$T_f : ~ g \mapsto f \, \cdot \, g$, associated to such a pointwise multiplier, must be continuous in $X$,
see \cite[p.~33]{MS2}.
We shall call $X$ an algebra with respect to pointwise multiplication 
(for short a multiplication algebra) if $f\, \cdot \, g \in X$ for all $f,g\in X$.
In addition we put
\[
M(X):= \Big\{f:~ f\, \cdot \, g \in X \qquad \forall g\in X\Big\}
\]
and equip this set with the norm of the operator $T_f$, i.e.,
\[
\|\, f\, |M(X)\|:= \| \, T_f : ~ X \to X\| = \sup_{\|g|X\|\le 1}\, \| \, f \cdot g \, |X\|\, .
\]


\subsection{Pointwise multipliers for Sobolev spaces}\label{main1}


One of our main results is as follows.

\begin{theorem}\label{main-so}
Let $m\in \N$ and $1<p<\infty$. Then $S^m_pW(\R)$ is a multiplication algebra.
\end{theorem}

\begin{remark}\rm
 For the isotropic case we refer to Moser \cite{Mo}, Strichartz \cite{Str-67} and the comprehensive monographs \cite{MS1}, \cite{MS2} of 
Maz'ya and Shaposnikova.
\end{remark}

One way to extend this result to $p=\infty$ is given by considering $C^m_\mix (\R)$ instead of $S^m_\infty W(\R)$.

\begin{definition}\label{def-co} Let $m\in \N$. Then  $C^m_\mix (\R)$ is the collection of all continuous functions 
$f:~\R\to \C $
such that all  derivatives $D^\alpha f$ with $\max_{j=1, \ldots \, d} \, \alpha_j\le m$ are continuous as well 
and
\beqq
\|\, f\, |C^m_\mix(\R)\| := \sum_{|\alpha|_{\infty} \leq m} \, \sup_{x \in \R}\, |\, D^{\alpha}f(x)\, |< \infty \, .
\eeqq 
\end{definition}

In this case the following result is almost trivial.

\begin{theorem}\label{cm}
Let $m\in \N$. Then $C^m_{\mix}(\R)$ is a multiplication algebra.
\end{theorem}

\noindent
We continue with a comment to Moser-type inequalities. 
Let $d=2$. Then, for $\alpha = (m,m)$ we obtain
\be\label{ws-014}
D^\alpha (f\, \cdot \, g) (x,y) = \sum_{j=0}^m \binom{m}{j} \sum_{\ell =0}^m \binom{m}{\ell}\, 
\frac{\partial^{j + \ell}f}{\partial x^\ell \, \partial y^j } (x,y)\,
\frac{\partial^{2m - (j + \ell)}g}{\partial x^{m-\ell} \, \partial y^{m-j} } (x,y) \, .
\ee
By choosing $j=0$ and $\ell=m$ we see that the term
$\frac{\partial^{m}f}{\partial x^m} \, \frac{\partial^{m}g}{\partial y^m } $ occurs in the previous sum.
Hence, Gagliardo-Nirenberg-type  inequalities  can not be applied as it is done in the isotropic case.
This is the main reason why we can not expect Moser-type inequalities for the dominating mixed case.

\begin{theorem}\label{negativea}
Let $d>1$ and $m\in \N$. \\
{\rm (i)} Then there exists no constant $C>0$ such that
\beqq
\|\, f\, \cdot \, g\, | C^m_\mix (\R)\| \leq C\, \big(\|\, f\, | C^m_\mix (\R)\|\cdot  
\|\, g\, |L_{\infty}(\R)\| + \|\, f\, |L_{\infty}(\R)\|\cdot \|g| C^m_\mix (\R)\|\big)
\eeqq
holds for all $f,g\in  C^m_\mix (\R)$.\\
{\rm (ii)} Let $1< p< \infty$. There exists no constant $C>0$ such that
\beqq
\|fg| S^m_pW(\R)\| \leq C\big(\|f| S^m_pW(\R)\|\cdot \|g|L_{\infty}(\R)\| + \|f|L_{\infty}(\R)\|\cdot\|g| S^m_pW(\R)\|\big)
\eeqq
holds for all $f,g\in  S^m_pW(\R)$.\\
\end{theorem}

Based on Theorems \ref{main-so}, \ref{cm} it is quite easy to get a characterization of 
$M(S^m_pW (\R))$ and $M(C^m_\mix (\R))$.

Let $\psi$ be a non-negative $C_0^{\infty}(\R)$ function. We put $
\psi_{\mu}(x)=\psi(x-\mu)$, $\mu\in \Z,\ x\in \R
$
and assume that
\be\label{ws-10}
\sum_{\mu\in \Z} \psi_{\mu}(x)=1\qquad  \text{ for all}\ x\in \R\,.
\ee

\begin{definition}\label{def-unif}  
Let the Banach space $X$ be continuously embedded into $ L_1^{\ell oc}(\R)$.
\\
{\rm (i)}  $X^{\ell oc} $ is the collection of all $g \in L_1^{\ell oc}(\R)$ such that 
$\varphi \, \cdot \, g \in X$ for all test functions $\varphi \in C_0^\infty (\R)$.
\\
{\rm (ii)} 
Let  $\psi$ be as in \eqref{ws-10}. 
Then $X_{\unif}$ is the collection of all $f\in X^{\ell oc}$ such that
\beqq
\|\, f\, | X_{\unif}\|_{\psi} = \sup_{\mu\in \Z} \, \|\, \psi_{\mu}\, \cdot \, f\, |  X\|<\infty.
\eeqq
\end{definition}

\begin{remark}\rm 
The spaces $ S^m_p W(\R)_{\unif}$ and $ S^r_{p,p} B(\R)_{\unif}$ are independent of the special choice of $\psi$ 
(in the sense of equivalent norms). These are consequences of Theorem \ref{cm} and Theorem \ref{main-be} respectively. 
\end{remark}

Now we are in position to formulate the main result of our paper.

\begin{theorem}\label{mul-space}
Let $d>1$ and  $m\in \N$.
\\
{\rm (i)} Let $1<p<\infty$. We have
\beqq
M(S^m_pW(\R))=S^m_{p}W(\R)_{\unif}
\eeqq 
in the sense of equivalent norms.
\\
{\rm (ii)}
We have
\beqq
M(C^m_\mix(\R))=C^m_{\mix}(\R)
\eeqq 
in the sense of equivalent norms.
\end{theorem}

\begin{remark}\rm
 For the classical case of isotropic Sobolev spaces  we refer again to Strichartz \cite{Str-67} and the  monographs \cite{MS1}, \cite{MS2} of 
Maz'ya and Shaposnikova.
\end{remark}


\subsection{Pointwise multipliers for Besov spaces}\label{main2}


The main  result  with respect to   Besov spaces of dominating mixed smoothness reads as follows.

\begin{theorem}\label{main-be}
Let $1\leq p\leq \infty$ and $r>0$. Then $S^r_{p,p}B(\R)$ is a multiplication algebra if and only if either $1 <p\le \infty$ and  
$r>1/p$ or $p=1$ and $r \ge 1$.
\end{theorem}

\begin{remark}\rm
There is a rather long list of references concerning the isotropic case. Let us refer to 
Peetre \cite{Pe}, Triebel \cite{Tr78}, Maz'ya, Shaposnikova \cite{MS1}, \cite{MS2} and 
Runst, S. \cite{RS} to mention at least a few. 
\end{remark}

Based on Theorem \ref{main-be} it is now quite easy to prove the following.

\begin{theorem}\label{mul-spaceb}
 Let  either  $1 <p\leq \infty$ and $r>1/p$ or $p=1$ and $r\ge 1$. Then
\beqq
M(S^r_{p,p}B(\R))=S^r_{p,p}B(\R)_{\unif}
\eeqq
holds in the sense of equivalent norms.
\end{theorem}

Also in case of Besov spaces of dominating mixed smoothness there is no hope for Moser-type inequalities. 

\begin{theorem}\label{negative}
Let $d>1$, $1\leq p\leq \infty$ and $r>0$. Then there exists no constant $C>0$ such that
\beqq
\|fg| S^r_{p,p}B(\R)\| \leq C\, \big(\|f|S^r_{p,p}B(\R)\|\cdot \|g|L_{\infty}(\R)\| + \|f|L_{\infty}(\R)\|\cdot\|g|S^r_{p,p}B(\R)\|\big)
\eeqq
holds for all $f,g\in S^r_{p,p}B(\R)\cap C(\R)$.
\end{theorem}


\subsection{Pointwise multipliers for Sobolev-Besov spaces defined on domains}
\label{main3}


As a service for the reader we investigate the local situation as well, i.e., we consider 
 pointwise multipliers for  Sobolev and Besov spaces defined on the unit cube $\Omega=[0,1]^d$. 
For convenience  we introduce the  spaces under consideration by taking restrictions.

\begin{definition} \label{defomega}
{\rm (i)} Let $ 1< p < \infty$  and $m\in \N$. Then
   $S^m_pW(\Omega)$ is the space of all $f\in L_p(\Omega)$ such that there exists  $g\in
   S^m_pW(\R)$ satisfying $f = g|_{\Omega}$. It is endowed with the quotient norm
   $$
      \|\, f \, |S^m_pW(\Omega)\| = \inf \Big\{ \|g|S^m_pW(\R)\|~: ~ g|_{\Omega} =
      f \Big\}\,.
   $$   
   {\rm (ii)} Let $ 1\leq  p\leq \infty$  and $r>0$. Then
      $S^{r}_{p, p}B(\Omega)$ is the space of all $f\in L_p(\Omega)$ such that there exists   $g\in
      S^{r}_{p, q }B(\R)$ satisfying $f = g|_{\Omega}$. It is endowed with the quotient norm
      $$
         \|\, f \, |S^{r}_{p,p}B(\Omega)\| = \inf \Big\{ \|g|S^{r}_{p,p}B(\R)\|~:~ g|_{\Omega} =
         f \Big\}\,.
      $$
\end{definition}

We have the following lemma.
\begin{lemma}Let $ 1< p < \infty$  and $m\in \N$. Then $S^m_pW(\Omega)$ is the collection of $f\in L_p(\Omega)$ such that
\beqq
\|f|S^m_pW(\Omega)\|:= \sum_{|\alpha|_{\infty} \leq m}\|D^{\alpha}f|L_p(\Omega)\|<\infty\,,\ \alpha \in \N_0^d\,.
\eeqq
\end{lemma}

Our main results obtained in the previous subsections carry over to the local case.

\begin{theorem}\label{main-be-1}
{\rm (i)} Let $m\in \N$ and $1<p<\infty$. Then $S^m_pW(\Omega)$ is a multiplication algebra.\\
{\rm (ii)} Let $1\leq p\leq \infty$ and $r>0$. Then $S^r_{p,p}B(\Omega)$ is a multiplication algebra if  and only if 
either $1 <p\le \infty$ and $r>1/p$ or $p=1$ and $r\ge 1$.
\end{theorem}

Similar as in the global case Thm. \ref{main-be-1} can be turned into a characterizations of 
$M(S^m_pW(\Omega))$ and $M(S^r_{p,p}B(\Omega))$, respectively.

\begin{theorem}\label{mul-spacew}
{\rm (i)} 
Let $m\in \N$ and $1<p<\infty$. Then
\[
M(S^m_pW(\Omega))=S^m_{p}W(\Omega))
\] 
holds in the sense of equivalent norms.
\\
{\rm (ii)}
 Let  either  $1 <p\leq \infty$ and $r>1/p$ or $p=1$ and $r\ge 1$. Then
\[
M(S^r_{p,p}B(\Omega))=S^r_{p,p}B(\Omega)
\]
holds in the sense of equivalent norms.
\end{theorem}

Also in the local situation a Moser-type inequality does not hold.

\begin{theorem}\label{negativec}
Let $d>1$.
\\
{\rm (i)}  Let $1< p< \infty$ and $m \in \N$. There exists no constant $C>0$ such that
\beqq
\|f \, \cdot \, g| S^m_pW(\Omega)\| \leq C\big(\|f| S^m_pW(\Omega)\|\cdot \|g|L_{\infty}(\Omega)\| + \|f|L_{\infty}(\Omega)\|\cdot\|g| S^m_pW(\Omega)\|\big)
\eeqq
holds for all $f,g\in  S^m_pW(\Omega)$.\\
\\
{\rm (ii)}
Let 
$1\leq p\leq \infty$ and $r>0$. Then there exists no constant $C>0$ such that
\beqq
\|f\, \cdot \, g| S^r_{p,p}B(\Omega)\| \leq C\, \big(\|f|S^r_{p,p}B(\Omega)\|\cdot 
\|g|L_{\infty}(\Omega)\| + \|f|L_{\infty}(\Omega)\|\cdot\|g|S^r_{p,p}B(\Omega)\|\big)
\eeqq
holds for all $f,g\in S^r_{p,p}B(\Omega) \cap L_\infty (\R)$.
\end{theorem}


\section{Proofs}\label{proofs}



\subsection{Proof of the results in Subsection \ref{main1}}


To prepare the proof of Theorem \ref{main-so} we need the following lemma.

\begin{lemma}\label{hilfe}
Let $1 <p< \infty$ and $m \in \N$.
Let $\beta \in \N_0^d$ such that there exists some $L \in \N$, $L < d$, and $\beta = (m , m, \ldots \, ,  m, \beta_{L+1}, \ldots \, ,\beta_d)$ 
where $\max_{j=L+1, \ldots\, , d} \beta_j < m$.
Let $N \in \N$ such that $L\leq N \leq  d$.
Then there exists a constant $C$ such that
\[
\Big(\int_{\re^{N}} \sup_{x_{N+1}, \ldots \, , x_d \in \re}\, |D^\beta f (x)|^p \prod_{j=1}^{N} \, dx_j\Big)^{1/p} 
\le C \, \|\, f\, |S^m_p W(\R)\|
\]
holds for all $f \in S^m_p W(\R)$.
\end{lemma}

\bproof
Using the density of functions with compactly supported Fourier transform in $S^m_pW(\R)$
(which is a consequence of Proposition \ref{def-do}) we may assume that $\supp \cf f$ is compact.
Let $(\chi_{ k })_k$ be  the non-smooth decomposition of unity  defined in \eqref{chi}. It follows  
\be\label{decom}
 f(x) =\sum_{ k \in \N_0^d}\gf^{-1}[\chi_{ k }\gf  f](x)\,,\qquad x\in \R,
\ee
where  the sum on the right-hand side of \eqref{decom} has only a  finite number of nontrivial terms. Consequently we have
\beqq
D^{\beta}f(x) =\sum_{ k \in \N_0^d}\gf^{-1}[\chi_{ k }\gf D^{\beta}f](x)\,,\qquad x\in \R.
\eeqq
Let $\cf_n $ denote the Fourier transform on $\re^n$. 
Freezing $x_{1}, \ldots \, ,x_{N}$ and choosing $n= d - N$ we get as above
\beqq
D^{\beta}f(x) =\sum_{k_{N+1}, \ldots \, , k_d \in \N_0} \gf^{-1}_n [\chi_{k_{N+1}} \otimes \ldots \otimes \chi_{k_d}\gf_n D^{\beta}f](x)\,,\qquad x\in \R.
\eeqq
By making use of this identity,  triangle inequality and the Nikol'skij inequality, stated in  Proposition \ref{Nikolski}, we conlude
\beqq
& I:=  & 
\Big(\int_{\re^{N}} \sup_{x_{N+1}, \ldots \, , x_d \in \re}\, |D^\beta f (x)|^p \prod_{j=1}^{N} \, dx_j\Big)^{1/p}
\\
& \le &  \sum_{k_{N+1}, \ldots \, , k_d \in \N_0} \bigg(\int\limits_{\re^{N}}
\sup_{x_{N+1}, \ldots \, , x_d \in \re}\,
\bigg| \gf^{-1}_n [\chi_{k_{N+1}} \otimes \ldots \otimes \chi_{k_d}\gf_n D^{\beta}f](x) \bigg|^p \prod_{j=1}^{N} \, dx_j\bigg)^{1/p}
\\
&\le & c_1 \, \sum_{k_{N+1}, \ldots \, , k_d \in \N_0} \bigg(\prod_{j=N+1}^d 2^{\frac{k_j}{p}}\bigg)
\bigg(\int\limits_{\R} \bigg|\gf^{-1}_n [\chi_{k_{N+1}} \otimes \ldots \otimes \chi_{k_d}\gf_n D^{\beta}f](x)  \bigg|^p d  x\bigg)^{1/p}\,.
\eeqq
The Littlewood-Paley assertion, see Proposition \ref{def-do}, implies
\beqq
&& \hspace{-0.7cm}
\bigg(\int\limits_{\re^{N}} \bigg|\gf^{-1}_n [\chi_{k_{N+1}} \otimes \ldots \otimes \chi_{k_d}\gf_n D^{\beta}f](x)  \bigg|^p  \prod_{j=1}^{N} \, dx_j \bigg)^{1/p}
\\
&\le & c_2 \, \bigg(\int\limits_{\re^{N}} \bigg(\sum_{k_1, \ldots \, , k_{N} \in \N_0}
\bigg|\gf^{-1} [\chi_{k_1}  \otimes \ldots \otimes \chi_{k_d}\gf D^{\beta}f](x)  \bigg|^2 \bigg)^{p/2} \prod_{j=N+1}^{d} \, dx_j \bigg)^{1/p}
\eeqq
We define a multi-index $\alpha$ by taking $\alpha + \beta = (m, \ldots \, ,m)$.
Inserting this inequality in the previously obtained one we find

\beqq
I &\le & c_3  \, \sum_{k_{N+1}, \ldots \, , k_d \in \N_0} \bigg(\prod_{j=N+1}^d 2^{\frac{k_j}{p}}\bigg)
\bigg\|\bigg(\sum_{k_1, \ldots \, , k_{N} \in \N_0} \Big|\gf^{-1}\big[ \chi_{k}\gf D^{\beta}f\big]\Big|^2\bigg)^{1/2}\bigg|L_p(\R)\bigg\|
\\
 &= & c_3  \, \sum_{k_{N+1}, \ldots \, , k_d \in \N_0} \bigg(\prod_{j=N+1}^d 2^{k_j (\frac 1p -\alpha_j)}\bigg)
 \\
 && \times \quad
 \bigg\|\bigg(\sum_{k_1, \ldots \, , k_{N} \in \N_0}  \bigg( \prod_{j=N+1}^d 2^{2 k_j\alpha_j}\bigg) 
 \Big|\gf^{-1}\big[ \chi_{k}\gf D^{\beta}f\big]\Big|^2\bigg\}^{1/2}\bigg|L_p(\R)\bigg\|
 \\
 &\le & c_4\,   \bigg\|\bigg\{\sum_{k\in \N_0^d} \bigg( \prod_{j=N+1}^d 2^{2 k_j\alpha_j} \bigg) 
 \Big|\gf^{-1}\Big[ \chi_{k} \gf D^\beta f\Big]\Big|^2\bigg\}^{1/2}\bigg|L_p(\R)\bigg\|\,,
\eeqq
where we used 
\[
 \alpha_j \ge 1 > \frac 1p \, , \qquad j= N + 1, \ldots \, d\, .
\]
Let $\phi_0, \phi \in C_0^\infty (\re)$ be functions such that 
\beqq
\phi_0 (\xi)  =   1 \quad \mbox{on}\quad [-1,1]\qquad\text{and}\qquad 
\phi (\xi)  =   1 \quad \mbox{on}\quad \supp (\chi_1)\,.
\eeqq
For $j \in \N$ we put 
$\phi_j (\xi):= \phi (2^{-j+1}t)$ and $
\phi_{ k } := \phi_{{k_1}}\otimes \ldots \otimes
\phi_{{k_d}}$ if $  k  \in
\N_0^d$. Then it follows
\beqq
\bigg\|\bigg\{\sum_{k\in \N_0^d} && \hspace*{-0.7cm}\bigg( \prod_{j=N+1}^d 2^{2 k_j\alpha_j} \bigg) 
\Big|\gf^{-1}\Big[ \chi_{k}(\xi) \, \phi_k (\xi)\, \xi^\beta \, \gf f(\xi)\Big](\, \cdot \, )\Big|^2\bigg\}^{1/2}\bigg|L_p(\R)\bigg\|
\\
& = &   \bigg\|\bigg\{\sum_{k\in \N_0^d} 2^{2 |k|_1 m } \,  \Big|\gf^{-1}\Big[ \chi_{k} \, M_k \, \gf  f\Big]\Big|^2
\bigg\}^{1/2}\bigg|L_p(\R)\bigg\|\, , 
\eeqq
where 
\[
M_k (\xi)  :=  \phi_k (\xi)\, \bigg( \prod_{j=1}^{N} 2^{- k_jm} \xi_j^{\beta_j}\bigg)  \, 
\bigg( \prod_{j=N+1}^d 2^{ k_j(\alpha_j-m)} \xi_j^{\beta_j}\bigg)\, . 
\]
Observe that in case $k_j \ge 1$, $j=1, \ldots \, d$, we have
\beqq
\| \, M_k (2^k\, \cdot \, )\, |S^r_p W(\R)\| &  = & \Big(\prod_{j=1}^{N} 2^{k_j(\beta_j -m)}\Big) \, 
\Big( \prod_{j=N+1}^d 2^{ k_j(\alpha_j + \beta_j - m)}\Big) 
\Big\| \, \phi_1 (2\xi)\, \xi^\beta \,\Big|S^r_p W(\R)\Big\| < \infty
\eeqq
for any $r>0$. For the remaining $k$ a more or less obvious modification can be applied. 
Hence we find 
\[
\sup_{k \in \N_0^d} \, \| \, M_k (2^k\, \cdot \, )\, |S^r_p W(\R)\| < \infty 
\]
since 
\[
 |\beta|_\infty \le m \qquad \mbox{and}\qquad \alpha_j + \beta_j\le m\, , \quad j= N+1, \ldots \,, d\, . 
 \]
 But this is guaranteed by our assumptions.
Now Lemma \ref{mul1} yields
\[
I\,  \lesssim \,  \bigg\|\bigg(\sum_{k\in \N_0^d} 2^{2|k|_1m}\big|\gf^{-1}  \chi_{k} \gf  f \big|^2\bigg)^{1/2}\bigg|L_p(\R)\bigg\|
\]
which completes the proof.
\eproof

\vskip 0.3cm
\noindent
{\bf Proof of Theorem \ref{main-so}.} Let $f,g\in S^m_pW(\R)$. 
We shall use the norm given in Lemma \ref{equi}
\beqq
\|f\, \cdot \, g |S^m_pW(\R)\|^*:= \sum_{\gamma\in \{0,m\}^d}\, \|D^{\gamma}(f\,\cdot \, g)|L_p(\R)\|.
\eeqq
Using the density of functions with compactly supported Fourier transform in $S^m_pW(\R)$
 we may assume that $f$ and $g$ are $C^\infty$ functions.
Leibniz rule yields
\[
 D^{\gamma}(f\,\cdot \, g)(x) = \sum_{\beta \in \N_0^d: ~ 0 \le \beta \le \gamma} \, \binom{\gamma}{\beta}
D^{\beta}f(x) \, D^{\gamma - \beta} g  (x)\, .
\]
Let us assume $|\beta|_\infty < m$.
Then from the definition of $S^m_p W (\R)$ we derive $D^\beta f \in S^{m-|\beta|_\infty}_p W(\R)$ and 
Lemma \ref{emb} we conclude  $S^{m-|\beta|_\infty}_p W(\R) \hookrightarrow C (\R)$. Hence 
\beqq
\| \, D^{\beta}f \, D^{\gamma - \beta} g  \, |L_p(\R)\| \, & \leq & \, \| \, D^\beta f\, |C(\R)\|\,  \| D^{\gamma - \beta} g\, |L_p(\R)\|
\nonumber
\\
& \leq & \, c_1 \,  \| \,  D^\beta f\, |S^{m-|\beta|_\infty}_pW(\R)\| \, \|\, g\, |S^m_pW(\R)\|
\nonumber
\\
& \leq & \, c_1 \,  \| \,   f\, |S^{m}_pW(\R)\| \, \|\, g\, |S^m_pW(\R)\|\, , 
\eeqq
where $c_1 := \| Id: ~S^{m-|\beta|}_p W(\R) \to C (\R) \|$.
Of course, a similar argument can be applied if $|\gamma -\beta|_\infty < m$. It remains to deal with 
the situation  $|\beta|_\infty = |\gamma -\beta|_\infty = m$.
Without loss of generality we assume
$\beta= (m, \ldots , m , \beta_{L+1}, \ldots \, , \beta_{N}, 0 , \ldots  \, ,0)$ for some $L,N\in \N$ and 
\[
0 <  \beta_j <m \, , \qquad L+1 \le j \le N\, .
\]
But now we can use Lemma \ref{hilfe} and obtain
\beqq
&& \| D^{\beta}f \cdot D^{\gamma - \beta}g\, |L_p(\R)\|
\\
&\le & 
\bigg(\int\limits_{\re^{N}} \sup_{x_{N+1}, \ldots \, , x_d \in \re}
|D^{\beta} f(x)|^p \prod_{j=1}^{N} d x_j\bigg)^{1/p} \bigg(\int\limits_{\re^{d-N}}
\sup_{x_1, \ldots \, , x_{N} \in \re} \, |D^{\gamma - \beta}g(x)|^p \prod_{j=N+1}^{d} d x_j\bigg)^{1/p}
\\
& \le & C^2\,  \|\, f\, |S^m_p W(\R)\| \, \|\, g\, |S^m_p W(\R)\|  
\eeqq
which proves the claim. \qed

\vskip 0.3cm
\noindent
{\bf Proof of Theorem \ref{negativea}.} Here we can work with the same test functions as in proof of 
Theorem \ref{negative} below.
Since the B-case is a bit more complicated we give details in this situation.
\qed

\vskip 0.3cm
\noindent
{\bf Proof of Theorem \ref{mul-space}.} Let further $\psi $ be the function as in the Definition \ref{def-unif}. 
Also Sobolev spaces of dominating mixed smoothness satisfy a localization property of the following form: it holds
\beqq
\|f|S^m_{p}W(\R)\|\asymp \Big(\sum_{\mu\in \Z} \| \psi_{\mu}f|S^m_{p}W(\R)\|^p\Big)^{1/p}\,.
\eeqq 
Here $1 <p< \infty$ and $m \in \N_0$ (we identify $S^0_{p}W$ with $L_p$).
Let
$\phi\in C_0^{\infty}(\R)$ with $\phi\equiv 1 $ on support of $\psi$. Let $f\in S^m_pW(\R)$ and $g\in S^m_{p}W(\R)_{\unif}$.
Employing this localization principle and Theorem \ref{main-so} we obtain
\beqq 
\| f \, \cdot \, g| S^m_pW(\R)\| &\leq &  c_1 \, \Big(\sum_{\mu\in \Z} \| \psi_{\mu}\phi_{\mu} gf|S^m_{p}W(\R)\|^p\Big)^{1/p}
\\
&\leq&  c_2\,  \Big(\sum_{\mu\in \Z} \|  \psi_{\mu}  f|S^m_{p}W(\R)\|^p\cdot\| \phi_{\mu}  g |S^m_{p}W(\R)\|^p\Big)^{1/p} 
\\
&\leq & c_3 \, \| f| S^m_{p}W(\R)\| \cdot \sup_{\mu\in \Z}\| \phi_{\mu}g|S^m_{p}W(\R)\|.
\eeqq
Since cardinality of the set $D_{\mu}:=\{\nu\in \Z: \supp \phi_\mu \cap \psi_{\nu} \not=\emptyset \}$ is finite and independent of $\mu$, from Theorem \ref{main-so} we obtain  
$$ \| \phi_{\mu}g|S^m_{p}W(\R)\|= \Big\| \phi_{\mu}g\Big(\sum_{\nu\in D_{\mu}} \psi_{\nu}\Big)\Big|S^m_{p}W(\R)\Big\| \ \leq\ c\, \sup_{\nu\in \Z}\| \psi_{\nu}g|S^m_{p}W(\R)\|$$ which implies
\beqq
\| f \, \cdot \, g| S^m_pW(\R)\| &\leq &  c_4\, \| f| S^m_{p}W(\R)\| \cdot \sup_{\mu\in \Z}\| \psi_{\mu}g|S^m_{p}W(\R)\|.
\eeqq
Hence,
 $$ S^m_{p}W(\R)_\unif \hookrightarrow M(S^m_{p}W(\R)).$$
On the other hand, with  $g\in M(S^m_pW(\R))$, we derive 
\beqq
\| \psi_{\mu}g|S^m_{p}W(\R)\| 
& \leq & \|g| M(S^m_{p}W(\R))\|\cdot \| \psi_{\mu}|S^m_{p}W(\R)\| 
\\
& = & \|g| M(S^m_{p}W(\R))\|\cdot \| \psi |S^m_{p}W(\R)\|.
\eeqq
Consequently  
$$M(S^m_{p}W(\R))  \hookrightarrow S^m_{p}W(\R)_{\unif}$$ which completes the proof.
\qed


\subsection{Proof of the results in Section \ref{main2}}


\vskip 0.3cm
\noindent
{\bf Proof of Theorem \ref{main-be}}. {\it Step 1. } Let $r <m\leq r+1$. Since the norm  $\|\cdot\,|\,S^{r}_{p,p}B(\R)\|^{(m)}$ 
does not depend on $m>r$ in the sense of equivalent norms, see Lemma \ref{diff1}, we shall prove that  
\beqq
     \|f\, \cdot \, g\,|\,S^r_{p,p} B(\R)\|^{(2m)} \leq C \,  \| f|S^r_{p,p}B(\R) \| \cdot \| g|S^r_{p,p}B(\R) \|\,
\eeqq 
holds for all $f,g\in S^r_{p,p}B(\R)$. 
Taking into account Lemma \ref{emb1} we obtain
\beqq
\| f\, \cdot \, g\,|L_p(\R)\| \, \leq\, \| f|L_p(\R)\|\cdot \| g|C(\R)\|\, \leq\, \| f|S^r_{p,p}B(\R)\|\cdot \|g|S^r_{p,p}B(\R)\|.
\eeqq  
This inequality can be interpreted as the estimate needed for the term with $e= \emptyset$.
Next we need some identities for differences.
Note that if $\psi,\, \phi:\ \re \to \C$ and $m \in \N$ we have
\be\label{formular}
\Delta_h^{m}(\psi \phi)(x)=\sum_{j=0}^{m} \binom{m}{j} \, \Delta_h^{m-j}\psi(x+j h)\, \Delta_h^{j}\phi(x),\qquad x,h\in \re \, ,
\ee 
which can be proved by induction on $m$.
Let $e\subset [d]$, $e\not=\emptyset$ and recall the notation
\[
x \diamond y = (x_1\, \cdot \,y_1,...\,, x_d\, \cdot \,  y_d)\in \R\, 
\]
and
\[
\N_0^d(e) = \big\{k\in \N_0^d: \ k_i=0\ \text{if}\ i\not \in e\big\}.
\]
Then we derive from \eqref{formular} that
\be\label{mot}
\Delta_{h}^{2{\bar{m}},e}(f\, \cdot \,  g)(x)=\sum_{ u \in \N_0^d(e),\,|u|_\infty\leq 2m} \binom{2\bar{m}}{u}\, 
\Delta_{h}^{2 \bar{m} - u ,e}f(x+ u  \diamond h)
\Delta_{h}^{ u ,e}g(x)\,,\quad\, x, h\in \R\, , 
\ee 
holds. Here $2\bar{m}-u :=(2m-u_1,...,2m-u_d)$ and 
\[
\binom{2\bar{m}}{u} = \prod_{i \in e} \binom{2m}{u_i} \, .
\]
The main step of the proof will consists in  estimating  the terms
\beqq
S_{e,u} :=\bigg\{\sum\limits_{k \in \N_0^d(e)} 2^{r|k|_1 p}\Big(\sup_{|h_i| < 2^{-k_i}, i \in e} 
\big\| \Delta_{h}^{2 m - u ,e}f(\cdot+ u  \diamond h)\Delta_{h}^{ u ,e}g(\cdot)|L_p(\R)\big\|\Big)^p\bigg\}^{1/p}
\eeqq 
$e \not= \emptyset$, $u \in \N_0^d(e)$, $|u|_\infty \le 2m$, by considering some different cases.
\\
{\it Step 2.} The case $u_i \le m$ for all $i\in e$. Obviously we have 
\[
2m-u_i \ge m \, , \qquad i \in e\, . 
\]
Using a change of variables in the $L_p$-integral we obtain
\beqq
 \big\| \Delta_{h}^{2\bar{m} -u ,e}f(\cdot+u  \diamond h)\Delta_{h}^{u ,e}g(\cdot)|L_p(\R)\big\|  
 &\le & \big\| \Delta_{h}^{2\bar{m} -u ,e}f(\cdot+u  \diamond h)|L_p(\R)\big\|\,  \sup_{x\in \R} |\Delta_{h}^{u ,e}g(x)|\\
& \le & c_1 \| g |C(\R)\|\,  \big\| \Delta_{h}^{\bar{m} ,e}f(\cdot )|L_p(\R)\big\|\,. 
\eeqq
The embedding $ S^r_{p,p}B(\R) \hookrightarrow C(\R)$    implies 
\beqq
\sup_{ |h_i| < 2^{-k_i}, i \in e} \big\| \Delta_{h}^{2\bar{m} -u ,e} f(\cdot+u  \diamond h)\Delta_{h}^{u ,e} \, g(\cdot)\, |L_p(\R)\big\| 
& \le & c_1\,  \big\|g |C(\R)\big\| \,  \omega_{\bar{m}}^{e}(f,2^{-k })_p   
\nonumber
\\
&\le & c_2\, \big\|g |S^{r}_{p,p}B(\R)\big\|\,   \omega_{\bar{m}}^{e}(f,2^{-k })_p  \,.
\eeqq
Consequently we have 
\beqq
S_{e,u}
& \le & c_2 \, \big\|g |S^{r}_{p,p}B(\R)\big\|  \, \bigg(\sum\limits_{k  \in \N_0^d(e)} 2^{r|k |_1 p}\, \omega_{\bar{m}}^{e}(f,2^{-k })^{p}_p\bigg)^{1/p}   
\nonumber
\\
& \le &  c_2\, \big\|g |S^{r}_{p,p}B(\R)\big\|\,  \big\|f |S^{r}_{p,p}B(\R)\big\|\, .
\eeqq
The case $u_i\geq m$ for all $ i\in e$ can be handled in the same way by interchanging the roles of $f$ and $g$.
\\
{\it Step 3.} The remaining cases. Let there exist $L ,N \in \N$ such that
$e= \{1,2,\ldots\, , N\}$,  $u \in \N_0^d(e)$  and 
\[
 u := (u_1, \ldots \, u_L,    u_{L+1},...,u_N,0,\ldots,0)
\]
with 
\[
m  \le  u_i \le 2 m\, , \quad i = 1, \ldots\, ,  L,\qquad 0\leq   u_i < m, \quad i = L+1, \ldots\, ,  N, 
\]
and $L  < N \le d$. By assuming $|u|_\infty > m $ we cover all remaining cases up to an enumeration.
\\
{\it Substep 3.1.} Let $r>1/p$.
Working with the tensor product system $(\varphi_{k})_{{k}\in \N_0^d}$  we conclude
\[
 f(x) = \sum_{ \ell \in \zz^{d}}  \gf^{-1} [ \varphi_{k+ \ell } \gf f](x)
\]
with convergence in $ S^r_{p,p}B(\re^d)$ and therefore in $C(\re^d)$. 
Here we used the convention that in the univariate case $\varphi_n \equiv 0$ if $n< 0$ which implies 
$\varphi_{(k_1, \ldots \, , k_d)} \equiv 0$ if $\min_j k_j < 0$.
Hence 
we have the decompositions
\beqq
f(x) = \sum_{{\ell} \in \zz^d}  \gf^{-1}  [\varphi_{{k} + {\ell} } \gf f](x)  \qquad 
\text{and}\qquad 
g(x)=\sum_{ {\nu} \in \zz^d}  \gf^{-1}  [\varphi_{{k} + {\nu} } \gf g](x)\, , \qquad x \in \R\, , 
\eeqq 
with convergence in $C(\R)$. To simplify notation we put
\[
f_{{\ell} } : =\gf^{-1} [\varphi_{ {\ell} } \gf f] \qquad \mbox{and} \qquad g_{{\ell} } : =\gf^{-1}  [\varphi_{{\ell} } \gf g]\, , \qquad 
\ell \in \zz^d\, .
\]  
Then we obtain from triangle inequality
 \beqq
 \big\| \Delta_{h}^{2 \bar{m} - u ,e}f(\cdot+ u  \diamond h)&&\hspace{-0.7cm}
 \Delta_{h}^{ u ,e}g(\cdot)\,  |L_p(\R)\big\|
\nonumber
 \\
 &\le & \sum_{  \ell , \nu \in \zz^d}  \big\| \Delta_{h}^{2 \bar{m} - u ,e}f_{k+ \ell }(\cdot+ u  \diamond h)
\Delta_{h}^{ u ,e}g_{k+ \nu }(\cdot  ) \, |L_p(\R)\big\|\,.
\eeqq
We will estimate the sum on the right-hand side term by term.
It follows
\beqq
\big\| \Delta_{h}^{2 \bar{m} - u ,e}f_{k+ \ell }(\cdot+ u  \diamond h)
&& \hspace{-0.8cm}\Delta_{h}^{ u ,e}g_{k+ \nu }(\cdot  )\,  |L_p(\R)\big\|
\nonumber
\\
& \le &  \bigg(\int\limits_{\re^{d-L}} \sup_{\substack{x_i\in \re\\i \le L}} 
\big|\Delta_{h}^{2 \bar{m} - u ,e}f_{k+ \ell }(x+ u  \diamond h) \big|^p \prod_{i=L+1}^d  d  x_i \bigg)^{1/p}    
\nonumber
\\
&& \times \qquad 
\bigg(\int\limits_{\re^{L}} \sup_{\substack{x_i\in \re\\i > L}} \big|\Delta_{h}^{ u ,e }g_{k+ \nu }(x) \big|^p
\prod_{i=1}^L  dx_i \bigg)^{1/p}
\eeqq
Let $\gf_L$ denote the Fourier transform with respect to $(x_1, \ldots  , x_L)$.
Observe that for any $h \in \re^L$
\[
\supp  \gf_L  (f_{k+\ell} (\, \cdot \, + h, x_{L+1} , \ldots  , x_d))\,  \subset \{(\xi_1, \ldots  , \xi_L):~ |\xi_j|\le 3 \, 2^{k_j+\ell_j-1}\, , \: j=1, \ldots \, , L \} \, ,
\]
independent of $x_{L+1} , \ldots \, x_d$. Consequently, Nikol'skijs inequality in Proposition \ref{Nikolski} yields
\beqq
 \bigg(\int\limits_{\re^{d-L}} \sup_{\substack{x_i\in \re\\i \le L}} && \hspace{-0.7cm}
\big|\Delta_{h}^{2 \bar{m} - u ,e}f_{k+ \ell }(x+ u  \diamond h) \big|^p \prod_{i=L+1}^d  d  x_i \bigg)^{1/p} 
\\
& \le &  c_3\, \prod_{i=1}^L 2^{(k_i + \ell_i)/p} 
\bigg(\int\limits_{\re^{d}}  
\big|\Delta_{h}^{2 \bar{m} - u ,e} f_{k+ \ell }(x+ u  \diamond h) \big|^p  dx \bigg)^{1/p} 
\eeqq
with a constant $c_3$ independent of $f$, $k$ and $\ell$. A simple change of coordinates  and an analogous argument with 
respect to $g_{k + \nu}$ results in 
\beqq
&& \hspace{-0.8cm}
\big\| \Delta_{h}^{2 \bar{m} - u ,e}f_{k+ \ell }(\cdot+ u  \diamond h)
\Delta_{h}^{ u ,e}g_{k+ \nu }(\cdot  )\,  |L_p(\R)\big\|
\\
& \le &  c_4 \, \Big(\prod_{i=1}^L 2^{(k_i + \ell_i)/p}\Big) \Big(\prod_{i=L+1}^{d} 2^{(k_i + \nu_i)/p}\Big) 
\big\|\Delta_{h}^{2 \bar{m} - u ,e}f_{k+ \ell } \big|L_p (\R)\big\|  \, \big\|\Delta_{h}^{u ,e}g_{k+ \nu} \big|L_p (\R)\big\|    
\nonumber
\eeqq
We need one more notation. We put
\[
 \omega (\ell):= \{i \in \{1, \ldots , d\}:~ \ell_i <0 \} \qquad \mbox{and}\qquad 
 \overline{\omega} (\ell):= \{i \in \{1, \ldots , d\}:~ \ell_i \ge 0 \}\, .
\]
Writing $\Delta_{h}^{2 \bar{m} - u ,e}$ as 
\[
\Delta_{h}^{2 \bar{m} - u ,e} = \Big(\prod_{i \in \overline{\omega} (\ell) \cap e} 
 \Delta_{h_i}^{2 m - u_i}\Big) \Big(\prod_{i \in \omega (\ell) \cap e } \Delta_{h_i}^{2 m - u_i}\Big)
\]
it is easily seen that 
\[
 \sup_{|h_i|< 2^{k_i} , ~ i \in e} \, \big\|\Delta_{h}^{2 \bar{m} - u ,e}f_{k+ \ell } \big|L_p (\R)\big\|
 \le c_5 \, \prod_{i \in \omega (\ell) \cap e} 2^{\ell_i (2m-u_i)}\, \big\|f_{k+ \ell } \big|L_p (\R)\big\|\, , 
\]
where we have applied Lemma \ref{ddim-1} and Proposition \ref{peetremax}.
Altogether we have found the estimate
\beq \label{ws-002}
 \sup_{|h_i|< 2^{k_i} , ~ i \in e} \, && \hspace{-0.8cm}
\big\| \Delta_{h}^{2 \bar{m} - u ,e}f_{k+ \ell }(\cdot+ u  \diamond h)
\Delta_{h}^{ u ,e}g_{k+ \nu }(\cdot  )\,  |L_p(\R)\big\|
\\
& \le &  c_6 \, \Big(\prod_{i=1}^L 2^{(k_i + \ell_i)/p}\Big) \Big(\prod_{i=L+1}^{d} 2^{(k_i + \nu_i)/p}\Big)  
\Big(\prod_{i \in \omega (\ell)\cap e } 2^{\ell_i (2m-u_i)} \Big)
\Big( \prod_{i \in \omega (\nu)\cap e} 2^{\nu_i u_i}\Big)\nonumber
\\
&& \qquad \times \quad  \|f_{k+ \ell } |L_p (\R)\|  \, \|g_{k+ \nu} |L_p (\R)\|    
\nonumber
\eeq
with a constant $c_6$ independent of $f,g,k,\ell$ and $\nu$.
Observe that
\beqq
&& \hspace{-0.8cm} 2^{r|k|_1}\, \Big(\prod_{i=1}^L 2^{(k_i + \ell_i)/p}\Big) \Big(\prod_{i=L+1}^{d} 2^{(k_i + \nu_i)/p}\Big)  
\Big(\prod_{i \in \omega (\ell) \cap e} 2^{\ell_i (2m-u_i)} \Big)
\Big( \prod_{i \in \omega (\nu)\cap e} 2^{\nu_i u_i}\Big)
\\
& = &   \Big(\prod_{i=1}^d 2^{(k_i + \ell_i) r}\,  2^{(k_i + \nu_i) r}\Big) 
\Big(\prod_{i=1}^{L} 2^{(k_i+\ell_i) ( \frac 1p -r)}\Big) \Big(\prod_{i=L+1}^d 2^{(k_i + \nu_i)(\frac 1p - r)} \Big)
\Big(\prod_{i=1}^L 2^{-\nu_i  r} \Big)
\\
& \times & \quad \Big(\prod_{i=L+1}^N 2^{-\ell_i r} \Big)\Big(\prod_{i=N+1}^d 2^{-\ell_i r} \Big)\Big(\prod_{i \in \omega (\ell)\cap e} 2^{\ell_i (2m-u_i)} \Big)
\Big( \prod_{i \in \omega (\nu)\cap e} 2^{\nu_i u_i}\Big)\, .
\eeqq
Later on we will have to sum up only with respect to those terms where $\min_j (k_j + \ell_j) \ge 0$
or $\min_j (k_j + \nu_j) \ge 0$. Observe that $k \in \N_0^d(e)$, i.e., $k_{L+1} = \,  \ldots \, = k_N=0$ and therefore 
$\ell_{N+1},  \,  \ldots \, , \ell_d \ge 0$. Taking this into account it is obvious that 
\beq\label{ws-003}
\Big(\prod_{i=N+1}^{d} 2^{-\ell_i r}\Big)&&\hspace{-0.89cm} \Big(\prod_{i=1}^{L} 2^{(k_i+\ell_i) ( \frac 1p -r)}\Big) 
\Big(\prod_{i=L+1}^d 2^{(k_i + \nu_i)(\frac 1p - r)} \Big)\\
&\le& 
\Big(\prod_{i=N+1}^{d} 2^{-(\ell_i+k_i) \varepsilon}\Big)\Big(\prod_{i=1}^{L} 2^{-(k_i+\ell_i) \varepsilon }\Big) 
\Big(\prod_{i=L+1}^d 2^{-(k_i + \nu_i) \varepsilon} \Big)
\nonumber
\le 1
\eeq
if  $\varepsilon=\min(r,r - 1/p)$. Let $\delta :=\min(r, m-r)$. Clearly $\delta \in (0,1)$. Furthermore 
\beqq
\Big(\prod_{i=L+1}^N 2^{-\ell_i r} \Big)
\Big(\prod_{i \in \omega (\ell)\cap e} 2^{\ell_i (2m-u_i)} \Big) 
&=& \Big(\prod_{L< i\leq N\atop i\in \bar{\omega}(\ell)} 2^{-\ell_ir} \Big)\Big(\prod_{L< i\leq N\atop i\in \omega(\ell)} 2^{\ell_i(2m-u_i-r)} 
\Big)\Big(\prod_{1\leq i\leq L\atop i\in \omega(\ell)} 2^{\ell_i(2m-u_i)} \Big)\\
& \le &   \Big(\prod_{i= L+1}^N  2^{-|\ell_i| \delta}\Big)
\eeqq
and
\beqq
\Big(\prod_{i=1}^L 2^{-\nu_i  r} \Big)  
\Big( \prod_{i \in \omega (\nu)\cap e} 2^{\nu_i u_i}\Big)
& = &  \Big(\prod_{{L< i \le N \atop i \in \omega (\nu)}}  2^{\nu_i u_i}\Big)
 \Big(\prod_{{i \le L \atop i \in \omega (\nu)}}  2^{\nu_i (u_i-r)}\Big)
\Big(\prod_{{i \le L \atop i \in \bar{\omega} (\nu)}}  2^{-\nu_i r }\Big)
\\
&\le &  
\Big(\prod_{i=1}^L  2^{-|\nu_i| \delta}\Big) \, .
\eeqq
Next we apply the inequality 
\beqq
\sum_{j\in \N_0} |a_j| \leq c_7 \, \Big(\sum_{j\in \N_0} 2^{j\varepsilon p}|a_j|^p\Big)^{1/p}\, , 
\eeqq
valid for all  $\varepsilon>0$ with an appropriate constant $c_7$ depending on $\varepsilon$. 
This yields
\beq \label{ws-004}
&& \hspace{-0.8cm}
\Bigg\{ \sum_{k \in \N_0^d(e)} \bigg[
\sum_{\substack{\ell_i\in \zz, i\not \in \{L+1,...,N\} \\ \nu_i\in \zz, L+1 \le i \le d}}
 \sup_{|h_i|< 2^{k_i} , ~ i \in e} \, 
\big\| \Delta_{h}^{2 \bar{m} - u ,e}f_{k+ \ell }(\cdot+ u  \diamond h)
\Delta_{h}^{ u ,e}g_{k+ \nu }(\cdot  )\,  |L_p(\R)\big\|\bigg]^p\Bigg\}^{1/p}
\\
& \le &
c_7 \, \Bigg\{ \sum_{k \in \N_0^d(e)} 
\sum_{\substack{\ell_i\in \zz, i\not \in \{L+1,...,N\} \\ \nu_i\in \zz, L+1 \le i \le d}}
 \sup_{|h_i|< 2^{k_i} , ~ i \in e} \, \big\| \Delta_{h}^{2 \bar{m} - u ,e}f_{k+ \ell }(\cdot+ u  \diamond h)
\Delta_{h}^{ u ,e}g_{k+ \nu }(\cdot  )\,  |L_p(\R)\big\|^p\Bigg\}^{1/p}
\nonumber
\\
& \le &
c_8 \, \Big(\prod_{i=1}^L  2^{-|\nu_i| \delta}\Big)
\Big(\prod_{i= L+1}^N  2^{-|\ell_i| \delta}\Big)
\Bigg\{ \sum_{k \in \N_0^d(e)} 
\sum_{\substack{\ell_i\in \zz, i\not \in \{L+1,...,N\} \\ \nu_i\in \zz, L+1 \le i \le d}} 
 2^{|k+\ell|_1rp}  \, 2^{|k+\nu|_1rp}
\nonumber
\\
&& \hspace{6cm}\times \quad 
 \|f_{k+ \ell } |L_p (\R)\|^p  \, \|g_{k+ \nu} |L_p (\R)\|^p  \Bigg\}^{1/p}  \, , 
\nonumber
\eeq
see \eqref{ws-002}, \eqref{ws-003}.
To prepare the next estimate we try to reorganize the summation in the sum in $\{\, \ldots \,\}$.
Therefore we consider 
\beqq
S(\ell_{L+1}, \ldots \, , \ell_N, \nu_1, \ldots \, , \nu_L):=
\sum_{k \in \N_0^d(e)} 
\sum_{\substack{\ell_i\in \zz, i\not \in \{L+1,...,N\} \\ \nu_i\in \zz, L+1 \le i \le d}} 
 2^{|k+\ell|_1rp}  \, 2^{|k+\nu|_1rp} a_{k+ \ell }  \, b_{k+ \nu}\, , 
\eeqq
where $\ell_{L+1}, \ldots \, , \ell_N, \nu_1, \ldots \, , \nu_L$ are considered as fixed and $(a_n)_n$ and $(b_j)_j$
are sequence of nonnegative numbers such that $a_n = b_j =0$ if $\min_{i=1, \ldots \, ,d} n_i, \min_{i=1, \ldots \, , d} j_i <0$.
Now we turn to an investigation of the linear system of equations 
\beq\label{a-10}
n_i & := &  k_i+\ell_i \, , \qquad i \in \{1, \, \ldots \,  , N\}\, ,  
\\
\label{a-11}
j_i & := & k_i + \nu_i \, , \qquad i \in \{1, \ldots \, , N\}\, .
\eeq
Here we consider $\ell_1, \ldots , \ell_L, \nu_{L+1}, \ldots \, , \nu_N, k_1, \ldots , k_N$ as variables.
Obviously we are confronted with the following types of smaller systems of equations
\beqq
n_i & := &  k_i + \ell_i 
\\
j_i - \nu_i & := & k_i \, , \qquad i\in {1, \ldots \, L}\, ,   
\eeqq
and 
\beqq
n_{i} - \ell_i & := &  k_{i}  
\\
j_{i} & := & k_{i} + \nu_{i}\, , \qquad i \in \{L+1, \ldots \, , N\}\, . 
\eeqq
Because of 
\[
 \left|\begin{matrix}
 1 & 1 \\
 1 & 0 
\end{matrix}\right| = 1 \qquad \mbox{and}\qquad  \left|\begin{matrix}
 1 & 0 \\
 1 & 1 
\end{matrix}\right| = 1
\]
we know that the mapping 
\[
T(\ell_{L+1}, \ldots \, , \ell_N, \nu_1, \ldots \, , \nu_L): ~ 
(k_1, \ldots, k_N, \ell_1, \ldots \, , \ell_L, \nu_{L+1}, \ldots \, , \nu_N) \mapsto (n_1, \ldots, n_N, j_1, \ldots \, j_N)
\]
is one-to-one. 
It maps $\N_0^N \times \zz^N$ onto a certain 
subset $\mathcal{T}(\ell_{L+1}, \ldots \, , \ell_N, \nu_1, \ldots \, , \nu_L)$ of $\zz^{2N}$.
Recall, $k_{N+1} = \ldots = k_d=0$. We supplement our system \eqref{a-10}, \eqref{a-11} by $n_i := \ell_i$ and $j_i:= \nu_i$, 
$i=N+1, \ldots \, , d$. 
Hence, we can extend $T$ to an injective mapping defined on $\zz^{2d}$. We denote this extension by ${\mathbb T}$.
Let $n:= (n_1, \ldots, n_d)$ and $j:= (j_1, \ldots \, j_d)$. 
Now we can perform a change of summation according to $n= k+\ell$ and $j= k + \nu$ simultaneously. 
This leads to 
\beqq
S(\ell_{L+1}, \ldots \, , \ell_N, \nu_1, \ldots \, , \nu_L) & = &
\sum_{(n_1, \ldots \, , n_N, j_1 \ldots \, , j_N) \in \mathcal{T}(\ell_{L+1}, \ldots \, , \ell_N, \nu_1, \ldots \, , \nu_L)}
\\
&& \quad  \times \quad 
\sum_{n_{N+1}, \ldots \, , n_d =0}^\infty  \sum_{j_{N+1}, \ldots \, , j_d =0}^\infty 
2^{|n|_1rp}  \, 2^{|j|_1rp} a_{n}  \, b_{j}\, . 
\eeqq
But this implies
\beqq
S(\ell_{L+1}, \ldots \, , \ell_N, \nu_1, \ldots \, , \nu_L) \le 
\sum_{(n,j) \in \N_0^{2d}} 2^{|n|_1rp}  \, 2^{|j|_1rp} a_{n}  \, b_{j} \, .
\eeqq
Rewriting this inequality we have proved
\beq \label{ws-005}
\Bigg\{ \sum_{k \in \N_0^d(e)} && \hspace{-0.7cm}
\sum_{\substack{\ell_i\in \zz, i\not \in \{L+1,...,N\} \\ \nu_i\in \zz, L+1 \le i \le d}} 
 2^{|k+\ell|_1rp}  \, 2^{|k+\nu|_1rp} \|f_{k+ \ell } |L_p (\R)\|^p  \, \|g_{k+ \nu} |L_p (\R)\|^p  \Bigg\}^{1/p} 
\nonumber
\\ 
& \le & \Bigg\{  
\sum_{(n,j) \in \N_0^{2d}} 2^{|n|_1rp}  \, 2^{|j|_1rp} \, \|f_{n} |L_p (\R)\|^p  \, \|g_{j} |L_p (\R)\|^p\Bigg\}^{1/p} 
\nonumber
\\
&\le & \|f |S^r_{p,p}B(\R)\|\, \|g |S^r_{p,p}B(\R)\|\, .
\eeq
Now we are in position to estimate $S_{e,u} $ under the given restrictions.  
From \eqref{ws-004} and  \eqref{ws-005}  we derive 
\beqq
 S_{e,u} &\le & 
\Bigg\{ \sum_{k \in \N_0^d(e)} \bigg[\sum_{\ell \in \zz^d}\, \sum_{\nu \in \zz^d}
 \sup_{|h_i|< 2^{k_i} , ~ i \in e} \, 
\big\| \Delta_{h}^{2 \bar{m} - u ,e}f_{k+ \ell }(\cdot+ u  \diamond h)
\Delta_{h}^{ u ,e}g_{k+ \nu }(\cdot  )\,  |L_p(\R)\big\|\bigg]^p\Bigg\}^{1/p}
\\
& \le & \sum_{\substack{\ell_i\in \zz, L < i \le N\\ \nu_i\in \zz,  1\leq i \le L}}
\Bigg\{ \sum_{k \in \N_0^d(e)} \bigg[\sum_{\substack{\ell_i\in \zz, i\not \in \{L+1,...,N\} \\ \nu_i\in \zz, L+1 \le i \le d}}
 \, \ldots \, \bigg]^p\Bigg\}^{1/p}
\\
& \le & 
c_8 \,  \sum_{\substack{\ell_i\in \zz, L < i \le N \\ \nu_i\in \zz, 1\leq i \le L}}
\Big(\prod_{i=1}^L  2^{-|\nu_i| \delta}\Big)
\Big(\prod_{i= L+1}^N  2^{-|\ell_i| \delta}\Big)  \| \, f \, |S^r_{p,p}B(\R)\|\,  \| \, g \, |S^r_{p,p}B(\R)\|
\\
& \le & c_9  \, \| \, f \, |S^r_{p,p}B(\R)\|\,  \| \, g \, |S^r_{p,p}B(\R)\|
\eeqq
with $c_9$ independent of $f$ and $g$. This proves the claim in case $r>1/p$.
\\
{\em Substep 3.2} Let $p=1$ and $r=1$.
We shall use \eqref{ws-002} with $\varepsilon =0$ and obtain
\[
 S_{e,u} \le  
\sum_{k \in \N_0^d(e)} \sum_{\ell \in \zz^d}\, \sum_{\nu \in \zz^d}
 \sup_{|h_i|< 2^{k_i} , ~ i \in e} \, 
\big\| \Delta_{h}^{2 \bar{m} - u ,e}f_{k+ \ell }(\cdot+ u  \diamond h)
\Delta_{h}^{ u ,e}g_{k+ \nu }(\cdot  )\,  |L_1(\R)\big\|
\]
Now we can continue as in the previous substep.
\\
{\it Step 4.} Proof of the necessity of the restrictions. 
Let $r>0$ and $1\leq p \leq \infty$. 
Then the isotropic Besov space $B^r_{p,p}(\re)$ is an algebra if and only if 
either $r>1/p$ or $r=p=1$,  see  \cite[Thm.~2.6.2/1]{Tr78}, \cite[Thm.~2.8.3]{Tr83} or \cite[Thm.~4.6.4/1]{RS}. 
Hence, if either $r=1/p$ for some $1 <p< \infty$ or $0 <r<1/p$, $1 \le p <\infty$,  
there exist two sequences $\{f_n\}_{n\in \N}\subset B^r_{p,p}(\re)$ and $\{g_n\}_{n\in \N}\subset B^r_{p,p}(\re)$ such that
\beqq
\| f_n \, \cdot \, g_n| B^r_{p,p}(\re)\| \geq n\, \| f_n| B^r_{p,p}(\re)\|\,  \| g_n| B^r_{p,p}(\re)\|\,, \qquad n \in \N\, .
\eeqq
Let $\Psi\in C_0^{\infty}(\re)$, $\Psi\not\equiv 0 $. For $n\in \N$ and  $x=(x_1,...,x_d)\in \R$ we define the sequences
\beqq
F_n(x):=f_n(x_1)\cdot \Psi(x_2)\cdot...\cdot \Psi(x_d)\quad
\text{and}\quad
G_n(x)=g_n(x_1)\cdot \Psi(x_2)\cdot...\cdot \Psi(x_d)\,.
\eeqq
The cross-norm property of $ S^r_{p,p}B(\R)$
yields $\{F_n\}_{n\in \N}\subset S^r_{p,p}B(\R)$ and $\{G_n\}_{n\in \N}\subset S^r_{p,p}B(\R)$. Using 
the cross-norm property once again we find
\beqq
\|\, F_n\, \cdot \, G_n\, |S^r_{p,p}B(\R)\|& = &  \| \, f_n\, \cdot \, g_n\, |B^r_{p,p}(\re)\|\,  \|\Psi^2|B^r_{p,p}(\re)\|^{d-1}\\
&\geq &   n\, \| f_n| B^r_{p,p}(\re)\|\,  \| g_n| B^r_{p,p}(\re)\|\,  \|\Psi^2|B^r_{p,p}(\re)\|^{d-1}\, 
\eeqq
and
\beqq
\|F_n|S^r_{p,p}B(\R)\|\cdot \| G_n|S^r_{p,p}B(\R)\| = \| f_n| B^r_{p,p}(\re)\|\,  \| g_n| B^r_{p,p}(\re)\|\,\|\Psi|B^r_{p,p}(\re)\|^{2(d-1)}\, .
\eeqq
This obviously disproves that  $S^r_{p,p}B(\R)$ is  a multiplication algebra. 
\qed

\vskip 0.3cm
\noindent
{\bf Proof of Theorem \ref{negative}.} {\em Step 1.} Let $r> 1/p$.
It will be convenient for us to switch to an equivalent norm on $B^r_{p,p} (\re)$ given by
\[
\|h|B^r_{p,p}(\re)\|\asymp  \|\, h\, |L_p(\re)\| + \Big( \int_0^\infty \, t^{-rp} \,  \omega_m (h, t)^p \frac{dt}{t}\Big)^{1/p} . 
\]
with $m>r$ (compare with \eqref{besov}).
Obviously the two terms on the right-hand side have different homogeneity properties.
We have 
\[
\| h (\lambda \, \cdot \, )|L_p(\re)\| = \lambda^{-1/p} \, \| h |L_p(\re)\|
\]
and 
\[
\Big( \int_0^\infty \, t^{-rp} \,  \omega_m (h(\lambda \, \cdot \, ), t)^p \frac{dt}{t}\Big)^{1/p} = \lambda^{r-1/p}
\, \Big( \int_0^\infty \, t^{-rp} \,  \omega_m (h, t)^p \frac{dt}{t}\Big)^{1/p}\, , \qquad \lambda >0\, .
\]
Let $f\in C_0^{\infty}(\re)$ with 
$\supp f\subset [-2,2]$, $f(t)=1$ if $t \in [-1,1]$ and $\sup_t |f(t)| =1$. 
For $n\in \N$ we define $f_n(t): =f(2^nt)$, $t\in \re$.
Hence we have
\beqq
\|f_n|L_\infty (\re)\|=1 \qquad \text{and}\qquad
\|f_n|B^r_{p,p}(\re)\| \asymp 2^{n(r-1/p)}.
\eeqq
Let $g\in C_0^{\infty}(\re)$ such that  $g(t)=1$ if $t \in [-2,2]$ and $\sup_t |g(t)| =1$. It follows 
\beqq
\|g|L_\infty (\re)\| \asymp \|g|B^r_{p,p}(\re)\|\asymp 1 \qquad\text{and}\qquad 
\|f_ng|B^r_{p,p}(\re)\|   \asymp 2^{n(r-1/p)}.
\eeqq
For $x\in \R$ we put
\beqq
F_n(x):= f_n(x_1) \, \prod_{j=2}^d g(x_j)\qquad \text{and}\qquad G_n(x): = g(x_1)\, \cdot \,  f_n(x_2)\, \cdot 
\prod_{j=3}^d g(x_j)\, .
\eeqq
The crossnorm property and $f_n \, \cdot \, g = f_n $ imply that 
\beqq
\|F_n\, \cdot G_n|S^r_{p,p}B(\R)\| \asymp 2^{2n(r-1/p)}
\eeqq
and
\beqq
\|F_n|S^r_{p,p}B(\R)\|\cdot\|G_n|L_\infty (\R)\| = \|G_n|S^r_{p,p}B(\R)\|\cdot\|F_n|L_\infty (\R)\| \asymp 2^{n(r-1/p)} .
\eeqq
This proves the claim in case $r>1/p$.
\\
{\em Step 2.} Let $0 < r \le 1/p$. This time the argument can not rely on the differential dimension $r-\frac 1p$.
Our construction is oriented in the observation made after \eqref{ws-014}.
For $n \in \N$ we define  $\phi_n $ such that $\phi_n (t) = 1$ if $1/n \le t \le 1$, 
$\phi_n (t) = \phi_1 (t)$, $t \ge 1$ and $\supp \phi_n \subset [\frac{1}{2n}, 3/2]$.
Let $\varepsilon >0$. 
We put 
\[
 f_n (t):= \phi_n (t) \, t\, \sin (t^{-\varepsilon})\, , \qquad t>0\, . 
\]
For $t\le 0$ we put $f_n (t) = 0$.
\\
{\em Substep 2.1.} Let $0 < r < 1$. 
Then we assume in addition that 
$\phi_n$ is linear on $[\frac{1}{2n}, \frac 1n]$ and $[1,3/2]$, i.e., 
\[
\phi_n (t) = 2n (t-\frac{1}{2n})\, , \qquad \frac{1}{2n} \le t \le  \frac 1n\, , 
\]
and 
\[
\phi_n (t) = 2 \, (\frac 32 - t)\, , \qquad 1 \le t \le \frac 32\, .
\]

	$$ \begin{tikzpicture} 
	 \fill (0,0) circle (1.5pt);
	 \draw[->, ](0,0) -- (6.8,0);
	 \draw[->, ] (0,0) -- (0,3);
	 \draw (2,2.0) -- (4,2.0);
	 \draw (-0.05,2.0) -- ( 0.05, 2.0);
	 \node[below] at (0,0) {$0$};
	 \node [below] at (4,0) {$1$};
	 \node [left] at (0,2.0) {$1$};
	 \node [below] at (6,0) {$\frac{3}{2}$};
	 \node [below] at (6,0) {$\frac{3}{2}$};
	 \node [below] at (2,0) {$\frac{1}{n}$};
	 \node [below] at (1,0) {$\frac{1}{2n}$};
	 \node [left] at (0,2.8) {$\phi_n(t)$};
	 \node [below] at (6.8,0) {$t$};
	 \draw (6, -0.05) -- (6, 0.05);
	 \draw (4, -0.05) -- (4, 0.05);
	 \draw (2, -0.05) -- (2, 0.05);
	 \draw (1, -0.05) -- (1, 0.05);
	 \draw (1, 0) -- (2, 2.0);
	 \draw (4, 2.0) -- (6, 0);
\node [right] at (2.0,-1) {Figure 1};
	 \end{tikzpicture}$$

Altogether $f_n$ is Lipschitz for all $n$ and 
\[
\sup_t |f_n (t)|\le 1\, .
\]
To estimate the norm in $B^r_{p,p} (\re)$ we proceed by real interpolation.
First we investigate the norm in $W^1_p (\re)$.
By assuming $\varepsilon > 1/p$, $\varepsilon \neq 1+ 1/p$,  we conclude
\beqq
\| f_n'|L_p (\re)\| & \asymp &  \Big(\int_{1/n}^1 |\sin (t^{-\varepsilon}) - \varepsilon \, t^{-\varepsilon}\, \cos (t^{-\varepsilon})|^p dt\Big)^{1/p}
+ 2n \, \Big(\int_{1/(2n)}^{1/n} |\, t^{1-\varepsilon} \, |^p dt\Big)^{1/p}
\\
& \asymp & n^{\varepsilon  -1/p} \, .
\eeqq
Next we employ 
\[
(L_p(\re), W^1_p (\re))_{r, p} = B^r_{p,p} (\re)
\]
in the sense of equivalent norms, see, e.g., \cite[Chapt.~6]{BL} and \cite[2.4]{Tr83}.
The related interpolation inequality 
\[
\|f_n |B^r_{p,p}(\re)\|\lesssim \|f_n|L_p (\re)\|^{1-r} \, \| f_n | W^1_p (\re)\|^r
\]
yields 
\be\label{ws-016}
\|f_n |B^r_{p,p}(\re)\|\lesssim n^{(\varepsilon  -1/p)r}\, .
\ee
Employing the characterization by first order differences  of $B^r_{p,p}(\re)$ one can show that there 
exists some positive constant $c$ such that 
\be\label{ws-017}
\|f_n |B^r_{p,p}(\re)\|\ge c\,  n^{(\varepsilon  -1/p)r}\, , \qquad n \in \N\, .
\ee
This is a bit technical, one may take the proof of Lemma 2.3.1/1 in \cite{RS} as an orientation.
Now we are in position to define 
our test functions.
Let $\psi $ be a nontrivial $C_0^\infty $ function on $\re$ such that $\psi (t)=1$ if $t \in \supp f_n$ and  $\sup_t |\psi (t)| = 1$. Then we put
\[
F_n (x):= f_n (x_1)\, \prod_{j=2}^d \psi (x_j)\, , \qquad x \in \R\, , \quad n \in \N\,.
\]
and 
\[
G_n (x):= \psi (x_1)\, f_n (x_2) \, \prod_{j=3}^d \psi (x_j)\, , \qquad x \in \R\, , \quad n \in \N\,.
\]
From the cross-norm property we derive 
\[
 \|F_n |S^r_{p,p}B(\R)\|\asymp \|G_n |S^r_{p,p}B(\R)\| \asymp \|f_n |B^r_{p,p}(\re)\|\asymp   n^{(\varepsilon  -1/p)r} 
\]
and analogously 
\[
\|F_n \, \cdot \, G_n |S^r_{p,p}B(\R) \| \asymp  \|f_n |B^r_{p,p}(\re)\|^2 \asymp  n^{2(\varepsilon  -1/p)r}\, , \qquad n \in \N\, .
\]
In view of $\|F_n |L_\infty (\R)\|, ~ \|G_n |L_\infty (\R)\|\le 1$ this proves the claim.
\\
{\em Step 2.} Let $r=p=1$. We need to modify our construction. We will be rather sketchy.
In Step 1 $f_n$ was Lipschitz. This time we need to have the first derivative to be Lipschitz.
By smoothing $\phi_n$ in such a way that $\phi_n'$ is Lipschitz, see Figure 2,
we can prove 
\[
\| f_n''|L_1 (\re)\| \asymp n^{\varepsilon}\, ,
\]
or with other words 
\[
\| f_n|W^2_1 (\re)\| \asymp n^{\varepsilon}\, ,
\]
Now we proceed by using
\[
(L_1(\re), W^2_1 (\re))_{1/2, 1} = B^1_{1,1} (\re)\, .
\]

	$$ \begin{tikzpicture}  
	 \fill (0,0) circle (1.5pt);
	 \draw[->, ](0,0) -- (6.8,0);
	 \draw[->, ] (0,0) -- (0,3);
	 \draw (2,2.0) -- (4,2.0);
	 \draw (-0.05,2.0) -- ( 0.05, 2.0);
	 \node[below] at (0,0) {$0$};
	 \node [below] at (4,0) {$1$};
	 \node [left] at (0,2.0) {$1$};
	 \node [below] at (6,0) {$\frac{3}{2}$};
	 \node [below] at (6,0) {$\frac{3}{2}$};
	 \node [below] at (2,0) {$\frac{1}{n}$};
	 \node [below] at (1,0) {$\frac{1}{2n}$};
	 \node [left] at (0,2.8) {$\phi_n(t)$};
\node [below] at (6.8,0) {$t$};
	 \draw (6, -0.05) -- (6, 0.05);
	 \draw (4, -0.05) -- (4, 0.05);
	 \draw (2, -0.05) -- (2, 0.05);
	 \draw (1, -0.05) -- (1, 0.05);
	 \draw (1,0) .. controls (1.4,0) and (1.6, 2.0).. (2, 2.0);
	 \draw (4,2.0) .. controls (4.6,2.0) and (5.4, 0).. (6, 0);
	 \node [right] at (2.0,-1) {Figure 2};
	 \end{tikzpicture}$$
Repeating the above arguments we can prove the claim also in this situation.
\qed
\vskip 3mm
To characterize $M(S^r_{p,p}B(\R))$ we need the so-called localization property for the  Besov spaces $S^r_{p,p}B(\R)$.
For it's proof we need another characterization by  differences. 
This time we shall work with pure differences (not with associated moduli of smoothness). 

\begin{lemma}\label{red}
Let $ 1\le p \le \infty$ and $r>0$. Let $m \in \N$ be a natural number such that $m>r$.
A function $f \in L_p (\R)$ belongs to $S^r_{p,p}B(\R)$ if and only if 
\[
  T_e:= \bigg\{\int_{[-1,1]^{|e|}} \prod_{i \in e} |h_i|^{-rp} \big\|  \Delta_{  h}^{\bar{m},e} f(\cdot) \big|
L_p (\R)\big\|^p \prod_{i \in e} \frac{dh_i}{|h_i|}\bigg\}^{1/p} <\infty\, ,
 \]
for all $e\subset [d]$. It follows that 
\[
\| f |L_p(\R)\| + \sum_{e\subset [d], e \neq \emptyset} T_e
\]
is an equivalent norm on $S^r_{p,p}B(\R)$.
\end{lemma}

\begin{remark}\rm
 A proof of a slightly modified statement (integration with respect  to the components $t_i$ is taken on $(0,\infty)$, not on $(0,1]$)
can be found in \cite{U1}. The reduction to the case considered in Lemma \ref{red} can be done by standard arguments, we omit details. 
\end{remark}

\begin{proposition}\label{F-case}
Let $1\leq p \leq \infty$ and $r>0$. 
Let $\psi$ be a non-negative $C_0^{\infty}(\R)$ function such that
\be\label{ws-010}
\sum_{\mu\in \Z} \psi(x-\mu) = 1\qquad  \text{ for all}\ x\in \R\,.
\ee
 We put $\psi_{\mu}(x):=\psi(x-\mu)$, $\mu\in \Z,\ x\in \R$.
Then
\beqq
\|f|S^r_{p,p}B(\R)\|\asymp \Big(\sum_{\mu\in \Z} \| \psi_{\mu}f|S^r_{p,p}B(\R)\|^p\Big)^{1/p}\,
\eeqq 
holds for all $f\in S^r_{p,p}B(\R)$ (usual modification for $p=\infty$).
\end{proposition}

\begin{proof} We prove for $1\leq p<\infty$. The proof for $p=\infty$ is modification.\\
{\it Step 1.} We shall prove that 
\be \label{<<}
\|f|S^r_{p,p}B(\R)\| \lesssim   \Big(\sum_{\mu\in \Z} \| \psi_{\mu}f|S^r_{p,p}B(\R)\|^p\Big)^{1/p} \,
\ee 
holds for all $f\in S^r_{p,p}B(\R)$. Again we shall work with the characterization by differences.
Let $m$ be a natural number such that $r<m\leq r+1$. Then, applying \eqref{ws-010}, the compactness of the support of $\psi$ and 
$|h|_\infty \le 1$,  we conclude
\beqq 
\|f|S^r_{p,p}B(\R)\|^p 
&\lesssim &   \sum_{e\subset [d]}  \sum_{k\in \N_0^d(e)} 2^{r|k|_1p}\sup_{|h_i|< 2^{-k_i}, i\in e}
\Big\|  \sum_{\mu\in \Z}|\Delta_h^{m}(f\psi_{\mu})(\cdot)|\Big|L_p(\R)\Big\|^p 
\\
&\lesssim & \sum_{e\subset [d]}  \sum_{k\in \N_0^d(e)} 2^{r|k|_1p}\sup_{|h_i|< 2^{-k_i}, i\in e}   
\sum_{\mu\in \Z}\big\|\Delta_h^{m}(f\psi_{\mu})(\cdot) \big|L_p(\R)\big\|^p 
\\
& \lesssim & \sum_{\mu \in \Z} \| \psi_{\mu}f|S^r_{p,p}B(\R)\|^p.
\eeqq
This proves \eqref{<<}.\\
{\it Step 2.} We shall prove the reverse direction to \eqref{<<}. In some sense we will follow the same strategy as in proof of Theorem \ref{main-be}.
Within this step we will use the characterization of $S^r_{p,p}B(\R)$ given in Lemma  \ref{red}.
\\
{\it Substep 2.1.} Some preparations. Let $r <m\leq r+1$. 
Clearly, in case $e=\emptyset$ we have 
\beqq
\sum_{\mu\in \Z} \|f\psi_{\mu}   |L_p(\R) \|^p \, 
 \lesssim  \|f|L_p(\R)\|^p  \lesssim  \| f|S^r_{p,p}B(\R)\|^p.
\eeqq 
For $e\subset [d]$,  $e\not=\emptyset$ we use
\[
\Delta_{h}^{2{\bar{m}},e}(f \cdot \psi_\mu)(x)= 
\sum_{ u \in \N_0^d(e),\,|u|_\infty\leq 2m} \binom{2\bar{m}}{u}\, \Delta_{h}^{2 \bar{m} - u ,e}f(x+ u  \diamond h)
\Delta_{h}^{ u ,e} \psi_\mu (x)\,,\quad\, x, h\in \R\, , 
\] 
see \eqref{mot}. Recall $2\bar{m}-u :=(2m-u_1,...,2m-u_d)$. 
It remains to  estimate  the terms
\beqq
S_{e,u}:=  \bigg\{\sum_{\mu\in \Z} \int_{[-1,1]^{|e|}} \prod_{i \in e} |h_i|^{-rp} \big\|  \Delta_{h}^{2 \bar{m} - u ,e}f(\cdot+ u  \diamond h)
\Delta_{h}^{ u ,e} \psi_\mu (\cdot)\big|
L_p (\R)\big\|^p \prod_{i \in e} \frac{dh_i}{|h_i|}\bigg\}^{1/p}
.
\eeqq 
This will be done by using the same splitting into various cases as done  in the proof of Theorem \ref{main-be}.
\\
{\it Substep 2.2.} The case $u_i<m$ for all $i\in e$. 
By assumption $\psi$ has compact support 
and therefore $\supp \psi_\mu$ is contained in a cube $Q(\mu,c)$ with center in $\mu$ and sidelength  $c>0$. 
Because of $|h|_\infty \le 1$  we find 
\be\label{=0}
|\Delta_{h}^{ u ,e}\psi_{\mu}(x)|=0 \qquad \mbox{if}\qquad \|x-\mu\|_\infty  > R:=c + 2m\, .
\ee 
Obviously it holds
\be\label{u<m}
|\Delta_{h}^{2\bar{m} - u ,e}f(x+ u  \diamond h)\Delta_{h}^{ u ,e}\psi_{\mu}(x)|
\lesssim  \| \psi |C(\R)\| \,  |\Delta_{h}^{2\bar{m} - u ,e}f(x+ u  \diamond h)|\, .
\ee
Combining \eqref{u<m} and \eqref{=0} we derive 
\beqq
S_{e,u} &\lesssim &  
\bigg\{
\int_{[-1,1]^{|e|}} \prod_{i \in e} |h_i|^{-rp}  \sum_{\mu\in \Z} \big\|  \Delta_{  h}^{  2\bar{m} - u ,e}
f(\cdot+ u  \diamond h)\big|L_p (Q(\mu, R ))\big\|^p \prod_{i \in e} \frac{dh_i}{|h_i|}\bigg\}^{1/p}
\nonumber
\\
&\lesssim &  
\bigg\{
\int_{[-1,1]^{|e|}} \prod_{i \in e} |h_i|^{-rp}   \big\|  \Delta_{  h}^{  2\bar{m} - u ,e}
f(\cdot+ u  \diamond h)\big|L_p (\R)\big\|^p \prod_{i \in e} \frac{dh_i}{|h_i|}\bigg\}^{1/p}
\nonumber
\\
&\lesssim & \|f|S^r_{p,p}B(\R)\| \, .
\eeqq 
{\it Substep 2.3.} The case $u_i\geq m$ for all $i\in e$. Let $0 < \varepsilon < m-r$.
Directly from the definition of the spaces $S^{r+\varepsilon}_{\infty,\infty}B (\R)$ we derive the inequality 
\be\label{u>m}
\prod_{i \in e} | h_i|^{-(r+\varepsilon)} |\Delta_{h}^{ u ,e} \psi_{\mu}(x)|
\leq  \| \psi_\mu  |S^{r+\varepsilon}_{\infty,\infty}B(\R)\| = \| \psi  |S^{r+\varepsilon}_{\infty,\infty}B(\R)\|\, .
\ee
This inequality, combined with \eqref{=0}, results in 
\beq \label{three}
S_{e,u} &\lesssim &  
\bigg\{
\int_{[-1,1]^{|e|}} \prod_{i \in e} |h_i|^{\varepsilon p} \,  \sum_{\mu\in \Z} \big\| \Delta_{ h}^{  2\bar{m} - u ,e}f(\cdot+ u  \diamond h)
\big|L_p (Q(\mu, R))\Big\|^p \prod_{i \in e} \frac{dh_i}{|h_i|}\bigg\}^{1/p}
\nonumber
\\
&\lesssim &  \|f|L_{p}(\R)\|
\bigg\{
\int_{[-1,1]^{|e|}} \prod_{i \in e} \frac{dh_i}{|h_i|^{1-\varepsilon p}}\bigg\}^{1/p}
\nonumber
\\
&\lesssim &  \|f|S^r_{p,p}B(\R)\|\, .
\eeq 
{\it Step 3.} 
 The remaining cases. Let there exist $L ,N \in \N$ such that
 $e= \{1,2,\ldots\, , N\}$,  $u \in \N_0^d(e)$  and 
 \[
  u := (u_1, \ldots \,, u_L,    u_{L+1},...,u_N,0,\ldots,0)
 \]
 with 
 \[
 m  \le  u_i \le 2 m\, , \quad i = 1, \ldots\, ,  L,\qquad 0\leq   u_i < m, \quad i = L+1, \ldots\, ,  N, 
 \]
 and $L  < N \le d$. By assuming $|u|_\infty > m $ we cover all remaining cases up to an enumeration.
Let $e_1 := \{1, \ldots , L\}$ and $e_2:= e \setminus e_1$. 
As in \eqref{u>m} we conclude 
\beqq
|\Delta_{h}^{ u ,e}\psi_{\mu}(x)| & \lesssim  & \sup_{x \in \R} |\Delta_{h}^{ u ,e_1}\psi_{\mu}(x)|
\leq   \| \psi  |S^{r+\varepsilon}_{\infty,\infty}B(\R)\|\, 
\prod_{i \in e_1} | h_i|^{r+\varepsilon} \, .
\eeqq 
In a similar way as in \eqref{three} we obtain
\beqq
S_{e,u} &\lesssim&   
\bigg\{
\int_{[-1,1]^{|e|}} \Big(\prod_{i \in e} |h_i|^{-r}   \prod_{i \in e_1} |h_i|^{r+\varepsilon} \Big)^p
\sum_{\mu\in \Z} \big\|  \Delta_{  h}^{  2\bar{m} - u ,e}f(\cdot+ u  \diamond h)
\big| L_p (Q(\mu, R ))\big\|^p \prod_{i \in e} \frac{dh_i}{|h_i|}\bigg\}^{1/p}
\\
 &\lesssim &  \bigg\{
 \int_{[-1,1]^{|e|}}\Big( \prod_{i \in e_2} |h_i|^{-r}  \prod_{i \in e_1} |h_i|^{\varepsilon} \Big)^p
  \big\|  \Delta_{  h}^{  2\bar{m} - u ,e}f(\cdot+ u  \diamond h)
 \big| L_p (\R)\big\|^p \prod_{i \in e} \frac{dh_i}{|h_i|}\bigg\}^{1/p}.
\eeqq
Next we apply the elementary inequality 
\[
 \big\| \Delta_{ h}^{  2\bar{m} - u ,e}f(\cdot+ u  \diamond h)
\big| L_p (\R)\big\|
\lesssim \big\|  \Delta_{ h}^{m ,e_2}f(\cdot )
\big| L_p (\R)\big\|
\]
since $2m-u_i \ge m$ if $i \in e_2$.
Hence,  we get 
\beqq
S_{e,u} &\lesssim &
\bigg\{
 \int_{[-1,1]^{|e_2|}} \prod_{i \in e_2} |h_i|^{-rp}  
  \big\|  \Delta_{ h}^{m ,e_2}f(\cdot )
  \big| L_p (\R)\big\|^p \prod_{i \in e} \frac{dh_i}{|h_i|}\bigg\}^{1/p}\bigg\{
  \int_{[-1,1]^{|e_1|}}    \prod_{i \in e_1} \frac{dh_i}{|h_i|^{1-\varepsilon p}}\bigg\}^{1/p}
\nonumber
\\
& \lesssim & \|f|S^r_{p,p}B(\R)\|\, .
\eeqq
as a consequence of Lemma \ref{red}. This finishes the proof.\end{proof}

\vskip 0.3cm
\noindent
{\bf Proof of Theorem \ref{mul-spaceb}}.
By employing Proposition \ref{F-case}, Theorem \ref{main-be} and similar arguments 
as in the proof of Theorem \ref{mul-spacew} one obtains the claimed 
identity $M(S^r_{p,p}B(\R))=S^r_{p,p}B(\R)_{\unif}$. 
\qed


\subsection{Proof of the results in Section \ref{main3}}


By definition the positive results (sufficient conditions) carry over to the local case.
Concerning the necessary conditions it remains to observe that all test functions used in this context 
for the proof on $\R$ had compact support. From these remarks Theorem \ref{main-be-1}  and 
Theorem~\ref{negativec} follow.\\
Concerning the proof of Theorem \ref{mul-spacew} we remark that 
the embedding of $S^m_p W (\Omega) \hookrightarrow M(S^m_pW (\Omega))$ follows from the algebra property.
If we assume $f \in M(S^m_pW (\Omega)$ we conclude that
\[
 \| f\, \cdot \, g| S^m_p W (\Omega)\| \le c \, \|g|S^m_p W (\Omega)\|
\]
holds for all $g \in S^m_p W (\Omega)$. But the function $g =1$ belongs to $S^m_p W (\Omega)$. Hence, $f$
must be an element of $S^m_p W (\Omega)$.
Similarly we argue in case  of $S^r_{p,p}B(\Omega)$.
\qed


\end{document}